\documentclass[11pt]{article}

\usepackage[draft]{ed}
\usepackage{amsfonts}
\usepackage{amsmath}
\usepackage[]{algorithm2e}
\usepackage{amssymb}
\usepackage{epsfig}
\usepackage{bm}
\usepackage{enumerate}
\numberwithin{equation}{section} 

\usepackage{color}
\usepackage[normalem]{ulem}

\title{Schur complement Domain Decomposition Methods for the solution of multiple scattering problems}
\author{Michael Pedneault, Catalin Turc, Yassine Boubendir}

\newtheorem{theorem}{Theorem}[section]

\newtheorem{remark}[theorem]{Remark}
\newenvironment{proof}{\hspace{0.5cm} {\bf Proof.}}
{$\quad {}_\blacksquare$\vspace{0.3cm}}

\date{}
\newcommand{\triple}[1]{{\left\vert\kern-0.25ex\left\vert\kern-0.25ex\left\vert #1 
    \right\vert\kern-0.25ex\right\vert\kern-0.25ex\right\vert}}

\begin{document}
\maketitle
\begin{abstract}
  We present a Schur complement Domain Decomposition (DD) algorithm for the solution of frequency domain multiple scattering problems. Just as in the classical DD methods we (1) enclose the ensemble of  scatterers in a domain bounded by an artificial boundary, (2) we subdivide this domain into a collection of nonoverlapping subdomains so that the boundaries of the subdomains do not intersect any of the scatterers, and (3) we connect the solutions of the subproblems via Robin boundary conditions matching on the common interfaces between subdomains. We use subdomain Robin-to-Robin maps to recast the DD problem as a sparse linear system whose unknown consists of Robin data on the interfaces between subdomains---two unknowns per interface. The Robin-to-Robin maps are computed in terms of well-conditioned boundary integral operators. Unlike classical DD, we do not reformulate the Domain Decomposition problem in the form a fixed point iteration, but rather we solve the ensuing linear system by Gaussian elimination of the unknowns corresponding to inner interfaces between subdomains via Schur complements. Once all the unknowns corresponding to inner subdomains interfaces have been eliminated, we solve a much smaller linear system involving unknowns on the inner and outer artificial boundary. We present numerical evidence that our Schur complement DD algorithm can produce accurate solutions of very large multiple scattering problems that are out of reach for other existing approaches. 
   
\end{abstract}

\section{Introduction\label{intro}}

The numerical simulation of interaction of acoustic, electromagnetic, and elastic waves with large ensembles/clouds of scatterers, collectively referred to as multiple scattering, plays an important role in a variety of applied fields such as seismology, meteorology, remote sensing, and underwater acoustics, to name but a few. The excellent monograph of Martin~\cite{martin2006multiple} contains a comprehensive account of both theoretical and numerical developments in this field. 

While the direct extension of single scatterer solvers to multiple scatterers is in principle straightforward, solvers in the latter case are confronted by considerably larger-sized problems that exhibit increasingly worse conditioning properties which can be attributed to the need to resolve complicated multiple reflections between scatterers. Thus, Krylov subspace iterative solvers for the associated linear algebra problems typically require very large numbers of iterations. Although certain preconditioning strategies can alleviate this issue to some extent in the diffuse case (e.g. when the distances between scatterers are large with respect to the wavelength of the probing incident wave)~\cite{antoine2008numerical,antoine2012wide}, general purpose preconditioners that work effectively throughout the frequency range are difficult to construct for boundary integral solvers for multiple scattering problems. 

On account of the limitations recounted above, the solution of multiple scattering problems involving large ensembles of scatterers has been approached through various approximations that render the computations tractable yet do not control the errors incurred. One of the most popular approaches is the Lax-Foldy method~\cite{foldy1945multiple,lax1951multiple} in which a multiple scattering scheme is set up to account for contributions on any one of the scatterers by the rest of the scatterers wherein the scatterers are replaced by point isotropic scatterers. Another widely used algorithm for solution of multiple scattering problems is the T-matrix method pioneered by Waterman~\cite{waterman1965matrix}. The main idea in this method is to use particular solutions of Helmholtz equation to construct functional bases for incoming fields  and outgoing (i.e. radiative) fields and to assign an operator between incoming fields impinging on a given scatterer and fields scattered by it using decompositions in those incoming/outgoing bases. This operator describes completely the geometrical and material properties of a single scatterer.  Using the T-matrix framework, the solution of multiple scattering problems consists of combining the T-matrices for each individual scatterer in the ensemble in a large linear system.  Truncated T-matrices can be computed by null-fields methods~\cite{waterman1965matrix} or more reliably and whenever possible by boundary integral equation methods~\cite{gurel1992recursive,martin2006multiple,lai2014fast}. However, the T-matrix method that uses spherical multipole expansions suffers from numerical instabilities associated with fast growth of Hankel functions~\cite{martin2006multiple}, and it was only recently that robust bases functions for T-matrix methods have been proposed and analyzed~\cite{ganesh2012convergence}.

We approach the multiple scattering problem with Domain Decomposition Methods (DDM) which are divide and conquer strategies for solution of large-sized problems whose direct solutions is too costly or out of reach to existing resources. In a nutshell, DDM decompose the original problem (typically associated to a PDE) to be solved in a certain computational domain into subproblems associated to subdomains, so that each subproblem can be solved efficiently with existing methods. The solutions of each of these subproblems are interconnected via boundary conditions that reflect properties of the solution of the original problem. The latter solution is typically retrieved through a fixed point iterative procedure from the subproblem solutions~\cite{Depres,Nataf}. However, the rate of convergence of the fixed point iterations is very slow~\cite{boubendirDDM}. In order to accelerate the speed of convergence of iterative DD algorithms, carefully designed transmission operators have been incorporated in the Robin data~\cite{Gander1,boubendirDDM}. 

We apply the DD strategy to multiple scattering problems by enclosing the ensemble/cloud of  scatterers in a domain bounded by an artificial boundary, and we proceed by subdividing this domain into a collection of nonoverlapping subdomains so that the (artificial) boundaries of the subdomains do not intersect any of the scatterers. The original scattering problems is thus decomposed into a sequence of multiple scattering subproblems in each of the subdomains. Following the common practice in DD methods for wave problems we connect the solutions of the subproblems via Robin boundary conditions matching on the common interfaces between subdomains~\cite{Depres}. Our DD approach is a direct solver that uses subdomain Robin-to-Robin maps defined as the operators that return outgoing Robin data on the boundary of the subdomain corresponding to solutions of the Helmholtz equation in that subdomain with (a) relevant physical boundary conditions on the scatterers included in the subdomain and (b) incoming Robin boundary conditions on the boundary of the subdomain. We use these Robin to Robin maps associated to each of the subdomains to recast the DD formulation for the solution of the multiple scattering problem into the form of a linear system whose unknown consists of global Robin data defined on the interfaces between subdomains---two unknowns per each interface. The matrix corresponding to this linear system has a block-sparse structure, the distributions of the populated blocks in the global matrix corresponding to the interconnectivity between the subdomains. Harkening back to ideas pertaining to nested dissection methods~\cite{George:1973:NDR} and multifrontal methods~\cite{Duff:1983:MSI:356044.356047} for the solution of sparse linear algebra problems related to finite difference/finite element discretizations, we solve the ensuing linear system by Gaussian elimination of the unknowns corresponding to inner interfaces between subdomains via Schur complements. We prove rigorously that the Schur complement elimination procedure does not break down. Once all the unknowns corresponding to inner subdomains interfaces have been eliminated, we reduce the original linear system of equations to a much smaller one involving unknowns on the inner and outer artificial boundary. Basically, if $\mathcal{O}(N)$ unknowns are needed for the solution of the global multiple scattering problem, our final stage linear system requires only $\mathcal{O}(N^{1/2})$ unknowns.

The idea of using Robin-to-Robin maps as robust alternative to the more popular Dirichlet to Neumann maps can be traced back to the work~\cite{Fliss2} where it was used to good effect for calculations involving  periodic waveguides containing defects/perturbations; see also~\cite{Fliss1} for a more recent application to computation of guided modes in photonic crystal waveguides. The ideas of using Schur complements for solution of DDM for wave propagation problems was presented in~\cite{bendali} in the context of scattering by deep cavities. The Schur complement elimination procedure that is central to our algorithm is equivalent to a hierarchical merging of the subdomains Robin-to-Robin maps to compute the global interior Robin-to-Robin map of the domain that contains inside the cloud of scatterers. The same idea was used in~\cite{Gillman1} for the solution of scattering problems in variable media, where subdomain spectral solvers are merged via Robin-to-Robin maps. This idea harkens back to the multidomain spectral solvers introduced in~\cite{orszag1980spectral,kopriva1998staggered,pfeiffer2003multidomain}. Similar ideas were used recently for multiple scattering problems~\cite{Bennetts1} by random arrays of circular scatterers where the authors merge subdomain (slabs in their case) solutions via Dirichlet-to-Neumann operators. The authors in~\cite{Bennetts1} refer to their algorithm as slab-clustering technique, and solve each slab (subdomain) problems with addition theorem multipole techniques for circular scatterers. Another application of DD Schur complement techniques can be found in computing in a stable manner the impedance of layered elastic media~\cite{norris2013stable}.

The central component of our algorithm is the use of Robin-to-Robin maps for subdomain problems that involve a collection of scatterers enclosed by an artificial boundary. The Robin data is exchanged on the artificial boundary and physically relevant boundary conditions are imposed on the scatterers, assumed to be homogeneous. We present a robust boundary integral operators based representation of the Robin-to-Robin maps that uses the regularization ideas developed in~\cite{turc1,turc_corner_N}. We show that the polynomially graded mesh Nystr\"om method introduced in~\cite{turc_corner_N,dominguez2015well,turc2016well} for discretization of Helmholtz boundary integral operators in Lipschitz domains leads to efficient calculations via direct solvers of subdomain Robin-to-Robin maps for two-dimensional multiple scattering problems. Once each subdomain Robin-to-Robin map is computed, we proceed with the hierarchical Schur complement elimination procedure that involves computing inverses of small and well conditioned matrices.  In the final stage of our algorithm we solve directly a linear system that involves interior and exterior Robin-to-Robin maps on the boundary of the domain that encloses the ensemble of scatterers. This last inversion turns out to be the dominant contributor to the computational cost of our algorithm: if $\mathcal{O}(N)$ discretization points are needed on the scatterers, the cost of our Schur complement DD algorithm is $\mathcal{O}(N^{3/2})$. More importantly, since we essentially construct a direct solver for multiple scattering problems, multiple incidences can be treated with virtually no additional overhead. We present numerical evidence that our Schur complement DD algorithm gives rise to important computational savings over direct methods for the solution of multiple scattering problems.

\parskip 2pt plus2pt minus1pt

\section{Domain decomposition approach for multiple scattering problems\label{MS}}

We consider the problem of scattering by an ensemble of multiple disjoint scatterers $S_{p}, p=1,\ldots,P$, that is find the scattered field $u^s$ such that
\begin{eqnarray}\label{scatt}
  \Delta u^s +k^2u^s&=&0\ {\rm in}\ \mathbb{R}^2\setminus \cup_{p=1}^P S_{p}\nonumber\\
  u^s+u^{inc}&=&0\ {\rm on}\ \partial S_{p},\ p=1,\ldots,P\nonumber\\
  \lim_{|r|\to\infty}r^{1/2}(\partial u^s/\partial r - iku^s)&=&0
\end{eqnarray}
where $k$ is a positive wavenumber and $u^{inc}$ is an incident field assumed to be a solution of the Helmholtz equation. The method of solution proposed in this paper can be extended to more general physical boundary conditions of the form $\mathcal{B}_{p}(u^s+u^{inc},\partial_{n_p} u^{s}+\partial_{n_p} u^{inc})=0$ on $\partial S_{p}$ (e.g. Neumann, mixed Dirichlet-Neumann, transmission) where the operators $\mathcal{B}_{p},\ p=1,\ldots,P$ are assumed to be linear and to give rise to well posed Helmholtz problems~\eqref{scatt}. In equation~\eqref{scatt}, the exterior unit normals to the domains $S_p$ are denoted by $n_p$. The scatterers $S_{p}$ are assumed to be either closed Lipschitz scatterers or open scatterers (e.g. cracks).

{\em Assumption: We assume that the collection of scatterers $S_{p}, p=1,\ldots,P$ is contained in a box $B_0$ that is the union of $L$ non-overlapping boxes $B_j, j=1,\ldots,L$ such that a given box $B_j$ contains in its interior the scatterers $S_{j_1},\ldots,S_{j_Q}$, with $\cup_{j=1}^L\left(\cup_{q=1}^Q S_{j_q}\right)=\cup_{p=1}^P S_p$.  This assumption can be made more general by requiring that the box $B_0$ is a union of non-overlapping subdomains $\Omega_j$ such that none of the boundaries of those subdomains intersects one of scatterers. We also assume that the arrangement of boxes $B_j,j=1,\ldots,L$ is two-dimensional, that is there are points on the skeleton $\cup_{j=1}^L\partial B_j$ that belong to the boundaries of four distinct subdomains.}

A Domain Decomposition (DD) approach for the scattering problem~\eqref{scatt} consists of defining the subdomain solutions
\begin{equation}\label{uj}
  u_j:=(u^s+u^{inc})|_{B_j\setminus \cup_{q=1}^Q S_{j_q}},\ 1\leq j\leq L,\quad u_0:=u^s|_{\mathbb{R}^2\setminus B_0}
\end{equation}
with Robin boundary conditions matching on the common interfaces between the subdomains $B_j$. More precisely,  for two adjacent subdomains $B_j$ and $B_\ell$, with $1\leq j,\ell\leq L$ that share a common interface we denote by $n_j$ is the unit normal on $\partial B_j\cap \partial B_{\ell}$ pointing toward the domain $B_{\ell}$, and $n_{\ell}$ is the unit normal on $\partial B_{\ell}\cap \partial B_{j}$
pointing toward the domain $B_{j}$ respectively, so that $n_j=-n_{\ell}$ on $\partial B_{j}\cap \partial B_{\ell}$. We enforce the continuity of $u^s+u^{inc}$ and its normal derivative on the common interfaces between adjacent boxes $B_j$ and $B_{\ell}$
 \[
 u_j|_{\partial B_j\cap \partial B_{\ell}} = u_{\ell}|_{\partial B_j\cap \partial B_{\ell}}=u|_{\partial B_j\cap \partial B_{\ell}},\qquad \partial_{n_j} u_j|_{\partial B_j\cap \partial B_{\ell}} = -\partial_{n_{\ell}} u_{\ell}|_{\partial B_j\cap \partial B_{\ell}}=\partial_{n_j} u|_{\partial B_j\cap \partial B_{\ell}}
 \]
 in the classical form of Robin boundary conditions matching on the interfaces between subdomains 
 \begin{eqnarray}\label{boxj}
   \Delta u_j+k^2u_j&=&0\ {\rm in}\ B_j\setminus \cup_{q=1}^Q S_{j_q}\quad u_j=0\ {\rm on}\ \partial S_{j_q},\ q=1,\ldots,Q\\
   \partial_{n_j} u_j -i\eta\ u_j &=&-\partial_{n_{\ell}} u_{\ell}-i\eta\
u_{\ell}\ {\rm on}\ \partial B_{j}\cap \partial B_{\ell}\nonumber,\qquad \eta>0
 \end{eqnarray}
 for all $1\leq j,\ell\leq L$ such that the subdomains $B_j$ and $B_\ell$ share a common edge. For subdomains $B_j, 1\leq j\leq L$ that share an edge with $\partial B_0$ we use the additional Robin boundary data matching
 \begin{equation}\label{robin3}
    \partial_{n_j} u_j -i\eta\ u_j =-\partial_{n_{0}} (u^0+u^{inc})-i\eta\
(u^0+u^{inc})\ {\rm on}\ \partial B_{j}\cap \partial B_{0}
    \end{equation}
where $u^0$ is the solution to the following Helmholtz equation in $\mathbb{R}^2\setminus B_0$ 
\begin{eqnarray*}
  \Delta u_0 +k^2u_0&=&0\ {\rm in}\ \mathbb{R}^2\setminus B_0\nonumber\\
  \partial_{n_0} u_0 -i\eta u_0&=&-\partial_{n_j}(u_j-u^{inc})-i\eta (u_j-u^{inc})\qquad {\rm on}\ \partial B_0\cap \partial B_j\nonumber\\
  \lim_{|r|\to\infty}r^{1/2}(\partial u_0/\partial r - iku_0)&=&0
\end{eqnarray*}
and $\partial_{n_0}$ is the normal derivative on $\partial B_0$ with respect to the unit normal exterior $n_0$ to $ B_0$. Given that each of the subproblems in the subdomains $B_j,j=1,\ldots,L$ and $\mathbb{R}^2\setminus B_0$ are well posed (see Section~\ref{well-posed}), the DD formulation is equivalent to the original problem~\eqref{scatt}. Classicaly, the DD formulation is solved via fixed point iterations~\cite{Depres,Nataf}. However, the rate of convergence of iterative DD is very slow, a possible remedy being the use of carefully designed transmission operators~\cite{boubendirDDM} in matching of Robin data. In contrast, our DD approach computes the global data $g$ defined as
\[
g=\{g_{j\ell}:=(\partial_{n_j}u_j-i\eta\ u_j)|_{\partial B_j\cap\partial B_\ell},\ 0\leq j,\ell\leq L,\ meas(\partial B_j\cap\partial B_\ell)\neq 0\}
\]
through a {\em direct solver} of the linear system whose unknown is $g$
\begin{equation}\label{system_A}
  \mathcal{D}g=G
\end{equation}
which results from rewriting equations~\eqref{boxj} and~\eqref{robin3}. The matrix operator $\mathcal{D}$ can be written explicitly in terms of subdomain Robin-to-Robin (RtR) maps/operators which we show in Section~\ref{well-posed} to be well defined for all wavenumbers $k$. Indeed, to each of these subproblems we associate a RtR map
 \begin{equation}\label{RtRboxj}
   \mathcal{S}^j(\psi_j):=(\partial_{n_j} u_j+i\eta\ u_j)|_{\partial B_j}
 \end{equation}
 where $u_j$ is the solution of the following problem:
 \begin{eqnarray*}
   \Delta u_j+k^2u_j&=&0\ {\rm in}\ B_j\setminus \cup_{q=1}^Q S_{j_q}\quad u_j=0\ {\rm on}\ \partial S_{j_q},\ q=1,\ldots,Q\\
   \partial_{n_j} u_j-i\eta u_j&=&\psi_j\ {\rm on}\ \partial B_j.
 \end{eqnarray*}
 In addition to the RtR operators $\mathcal{S}^j$ we will make use of subdomain to scatterer Robin-to-Cauchy data operators
 \begin{equation}\label{eq:Yj}
   Y^j(\psi_j):=(u_j,\partial_{n_{j_q}} u_j)|_{\cup_{q=1}^Q S_{j_q}}.
   \end{equation}
 We note that knowledge of subdomain Robin data $g$ and the operators $Y^j$ allows us to compute the solution of the problem~\eqref{scatt} via Green's identities. In order to make the notation more suggestive, we will refer in what follows to the argument of the operator $\mathcal{S}^j$ defined in equation~\eqref{RtRboxj} in the form
 \[
 \mathcal{S}^j(\partial_{n_j}u_j-i\eta\ u_j)|_{\partial B_j}=(\partial_{n_j} u_j+i\eta\ u_j)|_{\partial B_j}.
 \]
 \begin{figure}
\centering
\includegraphics[width=60mm]{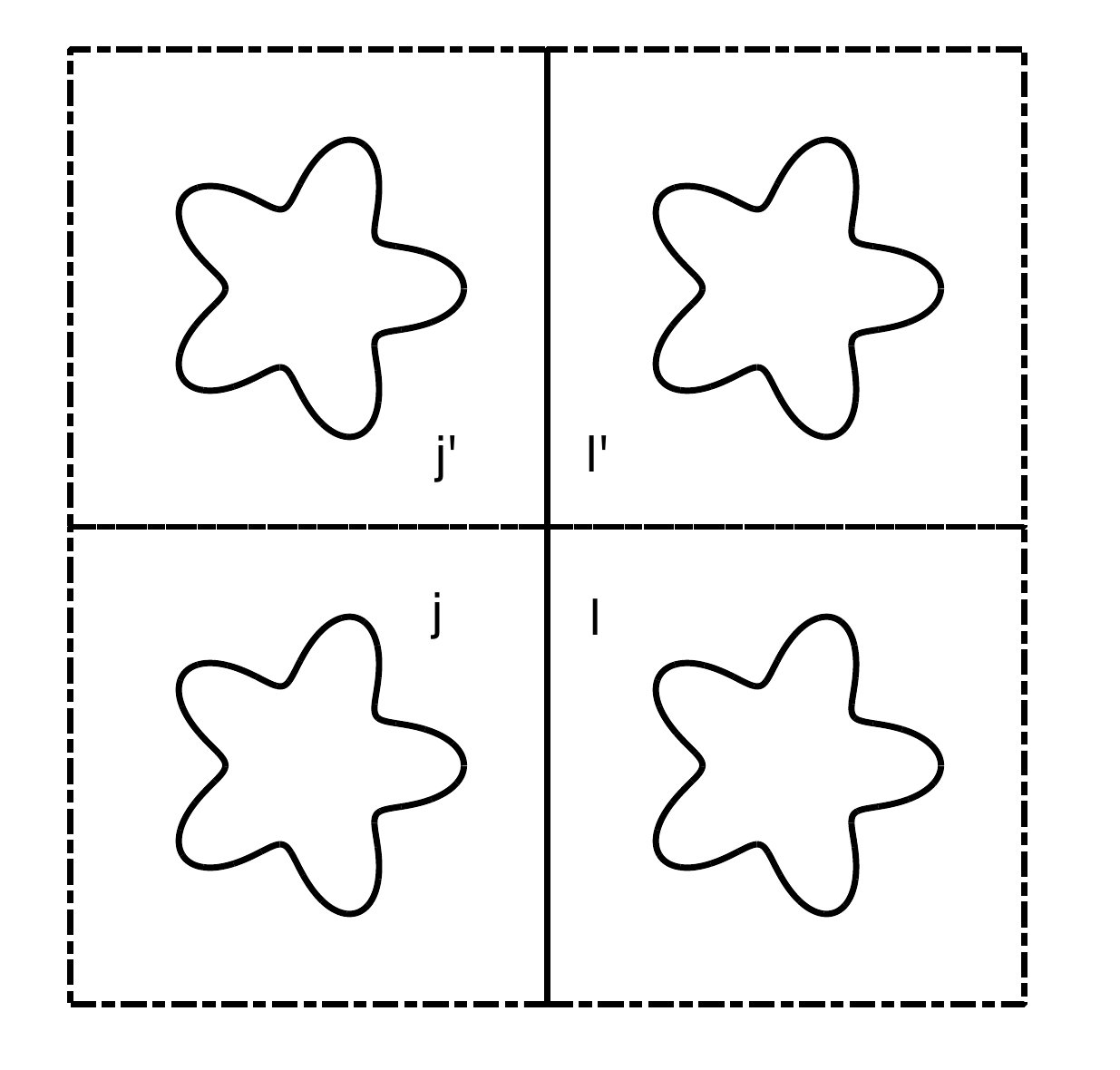}
\caption{A four subdomain configuration.}
\label{fig:4domains}
\end{figure}
With the aid of these operators, we show next how the inner interface data $(g_{j\ell},g_{\ell j})$, $(g_{j',\ell'},g_{\ell' j'}$, $(g_{jj'},g_{j'j})$, and $(g_{\ell\ell'},g_{\ell',\ell})$ corresponding to a four subdomain configuration depicted in Figure~\ref{fig:4domains} (we assume that this is a subset of a bigger subdomain ensemble and that none of the four subdomains has an edge in common with $\partial B_0$) can be eliminated via Schur complements from the linear system~\eqref{system_A}. To that end we define next interface subdomain RtR maps. For the sake of brevity we present these in the case of the subdomain $j$ in which case these maps amount to splitting the operator $\mathcal{S}^j$ in block form as
\[
\mathcal{S}^{j}=\begin{bmatrix}\mathcal{S}^j_{\ell j,j\ell}&\mathcal{S}^j_{\ell j,jj'} & \mathcal{S}^{j}_{\ell j,j\hat{j}}\\ \mathcal{S}^j_{j'j,j\ell}& \mathcal{S}^j_{j'j,jj'} & \mathcal{S}^j_{j'j,j\hat{j}} \\ \mathcal{S}^j_{\hat{j}j,j\ell}& \mathcal{S}^j_{\hat{j}j,jj'} & \mathcal{S}^j_{\hat{j}j,j\hat{j}} \end{bmatrix}
\]
so that the block components of the operator $\mathcal{S}^j$ have the precise definition
\begin{equation}\label{split1}
\begin{bmatrix}\mathcal{S}^j_{\ell j,j\ell}&\mathcal{S}^j_{\ell j,jj'} & \mathcal{S}^{j}_{\ell j,j\hat{j}}\\ \mathcal{S}^j_{j'j,j\ell}& \mathcal{S}^j_{j'j,jj'} & \mathcal{S}^j_{j'j,j\hat{j}} \\ \mathcal{S}^j_{\hat{j}j,j\ell}& \mathcal{S}^j_{\hat{j}j,jj'} & \mathcal{S}^j_{\hat{j}j,j\hat{j}} \end{bmatrix}\begin{bmatrix}g_{j\ell}\\g_{jj'}\\g_{j\hat{j}}\end{bmatrix}=\begin{bmatrix}(\partial_{n_j}u_j+i\eta u_j)|_{\partial B_j\cap \partial B_\ell}\\(\partial_{n_j}u_j+i\eta u_j)|_{\partial B_j\cap \partial B_{j'}}\\(\partial_{n_j}u_j+i\eta u_j)|_{\partial B_{\hat{j}}}\end{bmatrix}
\end{equation}
where we denoted $\partial B_{\hat{j}}:=\partial B_j\setminus(\partial B_\ell\cup\partial B_{j'})$. Alternatively, the interface subdomain RtR maps can be defined by considering Helmoltz problems with Robin data that is equal to zero on the complement of the interface on the boundary of the subdomain. Reordering conveniently the interface unknowns $g_{j\ell}$ we present in detail the block of the matrix $\mathcal{D}$ featuring in the linear system~\eqref{system_A} from which the unknowns $(g_{j\ell},g_{\ell j})$, $(g_{j',\ell'},g_{\ell' j'}$, $(g_{jj'},g_{j'j})$, and $(g_{\ell\ell'},g_{\ell',\ell})$ are eliminated:
\begin{equation}\label{matrixjljp}
  \begin{bmatrix}
  I & \mathcal{S}^\ell_{j\ell,\ell j}& 0 & 0 & 0 & \mathcal{S}^\ell_{j\ell,\ell\ell'}& 0 & 0 &\ldots\\
  \mathcal{S}^j_{\ell j,j\ell} & I & 0 & 0 & \mathcal{S}^j_{\ell j, jj'}& 0 & 0 & 0 & \ldots\\
  0 & 0 & I & \mathcal{S}^{\ell'}_{j'\ell',\ell' j'} & 0 & 0 & 0 & \mathcal{S}^{\ell'}_{j'\ell',\ell'\ell} & \ldots\\
  0 & 0 & \mathcal{S}^{j'}_{\ell' j', j' \ell'}& I & 0 & 0 & \mathcal{S}^{j'}_{\ell' j', j'j}& 0 & \ldots\\
  0 & 0 & \mathcal{S}^{j'}_{j j', j'\ell'} & 0 & I & 0 & \mathcal{S}^{j'}_{j j', j'j} & 0 & \ldots\\
  0 & 0 & 0 & \mathcal{S}^{\ell'}_{\ell \ell', \ell' j'} & 0 & I & 0 & \mathcal{S}^{\ell'}_{\ell \ell', \ell' \ell}& \ldots\\
  \mathcal{S}^{j}_{j'j, j\ell} & 0 & 0 & 0 &  \mathcal{S}^{j}_{j'j, jj'} & 0 & I & 0& \ldots\\
  0 & \mathcal{S}^{\ell}_{\ell'\ell, \ell j} & 0 & 0 & 0 & \mathcal{S}^{\ell}_{\ell'\ell, \ell \ell'} & 0 & I & \ldots\\
  \ldots & \ldots & \ldots & \ldots & \ldots & \ldots & \ldots & \ldots & \ldots \\
  \end{bmatrix}
  \begin{bmatrix}g_{j\ell}\\g_{\ell j}\\g_{j'\ell'} \\ g_{\ell' j'}\\g_{j j'}\\g_{\ell\ell'}\\g_{j'j}\\g_{\ell'\ell}\\\ldots\end{bmatrix}=\begin{bmatrix}0\\0\\0\\0\\0\\0\\0\\0\\\ldots\end{bmatrix}.
  \end{equation}
The pairs of unknowns $(g_{j\ell},g_{\ell j})$ and $(g_{j'\ell'},g_{\ell',j'})$ can be eliminated simultaneously from the linear system~\eqref{matrixjljp} via Schur complements. To this end we define
\begin{equation}\label{Ajl}
  \mathcal{D}_{j\ell}=\begin{bmatrix}I & \mathcal{S}^\ell_{j\ell,\ell j}\\\mathcal{S}^j_{\ell j,j\ell} & I\end{bmatrix}
\end{equation}
whose inverse is given by
\begin{equation}\label{inv_matrix_explicit}
  \mathcal{D}_{j\ell}^{-1}=\begin{bmatrix}
I+\mathcal{S}^\ell_{j\ell,\ell j}(I-\mathcal{S}^\ell_{j\ell,\ell j}\mathcal{S}^j_{\ell j,j\ell})^{-1}\mathcal{S}^j_{\ell j,j \ell} & -\mathcal{S}^\ell_{j\ell,\ell j}(I-\mathcal{S}^\ell_{j\ell,\ell j}\mathcal{S}^j_{\ell j,j \ell})^{-1}\\
-(I-\mathcal{S}^\ell_{j\ell,\ell j}\mathcal{S}^j_{\ell j ,j \ell})^{-1}\mathcal{S}^j_{\ell j,j \ell} &
(I-\mathcal{S}^\ell_{j\ell,\ell j}\mathcal{S}^j_{\ell j,j \ell})^{-1}
\end{bmatrix}
\end{equation}
under the assumption that the operators $I-\mathcal{S}^\ell_{j\ell,\ell j}\mathcal{S}^j_{\ell j,j\ell}$ are invertible; similar considerations apply to the matrix counterpart $\mathcal{D}_{j'\ell'}$. Using the Schur complement of the matrix
\[
\begin{bmatrix} \mathcal{D}_{j\ell} & 0 \\ 0 &  \mathcal{D}_{j'\ell'}\end{bmatrix}
\]
in equations~\eqref{matrixjljp} we obtain
\begin{equation}\label{Schur1}
  \begin{bmatrix}I & \mathcal{S}^2_{j\ell j'\ell'} & \ldots \\ \mathcal{S}^1_{j\ell j'\ell'} & I & \ldots\\\ldots & \ldots& \ldots\end{bmatrix}\begin{bmatrix}g_{jj'\ell\ell'}\\g_{j'j\ell'\ell}\\\ldots\end{bmatrix}=\begin{bmatrix}0\\0\\\ldots\end{bmatrix}
  \end{equation}
where
\[
g_{jj'\ell\ell'}:=\begin{bmatrix}g_{jj'}\\g_{\ell\ell'}\end{bmatrix}\qquad g_{j'j\ell'\ell}:=\begin{bmatrix}g_{j'j}\\g_{\ell'\ell}\end{bmatrix}
\]
with
\begin{equation}\label{merge1}
  \mathcal{S}^1_{j\ell j'\ell'}=\begin{bmatrix}\mathcal{S}^j_{j'j,jj'}&0\\0&\mathcal{S}^\ell_{\ell'\ell,\ell\ell'}\end{bmatrix}-\begin{bmatrix} \mathcal{S}^{j}_{j' j, j\ell} & 0 \\ 0 & \mathcal{S}^{\ell}_{\ell' \ell, \ell j} \end{bmatrix}\mathcal{D}_{j\ell}^{-1}\begin{bmatrix}0 & \mathcal{S}^{\ell}_{j\ell,\ell\ell'} \\ \mathcal{S}^{j}_{\ell j, jj'} & 0\end{bmatrix}
  \end{equation}
  and
  \begin{equation}\label{merge2}
    \mathcal{S}^2_{j\ell j'\ell'}=\begin{bmatrix}\mathcal{S}^{j'}_{jj',j'j}&0\\0&\mathcal{S}^{\ell'}_{\ell\ell',\ell'\ell}\end{bmatrix}-\begin{bmatrix} \mathcal{S}^{j'}_{j j', j'\ell'} & 0 \\ 0 & \mathcal{S}^{\ell'}_{\ell \ell', \ell' j'} \end{bmatrix}\mathcal{D}_{j'\ell'}^{-1}\begin{bmatrix}0 & \mathcal{S}^{\ell'}_{j'\ell',\ell'\ell} \\ \mathcal{S}^{j'}_{\ell' j', j'j} & 0\end{bmatrix}.
  \end{equation}
  The pair of unknowns $(g_{jj'\ell\ell'},g_{j'j\ell'\ell})$, in turn, can be eliminated from the linear system~\eqref{Schur1} by applying yet one more Schur complement corresponding to the submatrix in the upper left corner of the matrix in equation~\eqref{Schur1}. Remarkably, both matrices $\mathcal{S}^1_{j\ell j'\ell'}$ and $\mathcal{S}^2_{j\ell j'\ell'}$ turn out to be subdomain interface RtR maps, as we explain in Section~\ref{RtRM}. Therefore, the upper left corner submatrix that features in equations~\eqref{Schur1} is of the same type as its counterpart in equations~\eqref{matrixjljp}, and thus the Schur elimination procedure is repeated in a recursive manner to eliminate all the unknowns corresponding to Robin data on all the subdomain interfaces that are in the interior of $B_0$. In the final stage of the algorithm the interior Robin data on $\partial B_0$ is connected to the exterior Robin data on $\partial B_0$ via the reduced linear system
\begin{equation}\label{system_A_reduced}
  \mathcal{D}_{reduced}\begin{bmatrix}g_{int,0}\\g_{0, ext}\end{bmatrix}=G_{reduced}.
  \end{equation}
Using the matching of $g_{int,0}$ and $g_{0,ext}$ on $\partial B_0$ through the exterior RtR map $\mathcal{S}_0$ and the incident field $u^{inc}$, the linear system~\eqref{system_A_reduced} can be further reduced to a half-sized linear system whose unknown is the interior Robin data $g_{0,ext}$. We mention that the Schur complement elimination is carried out in practice without storing the matrix $\mathcal{D}$ in the linear system~\eqref{system_A}. It is only the reduced matrix $\mathcal{D}_{reduced}$ that is stored in practice. Once the Robin data $g_{0,ext}$ is computed, backward substitution delivers all the interface Robin data $g_{j\ell}$. In order to compute the solution $u^s$ of the multiple scattering problem~\eqref{scatt}, we use for each subdomain operators that map the corresponding Robin subdomain data to Dirichlet and/or Neumann boundary data on the scatterers. The latter operators can be computed as byproducts of computations of RtR subdomain maps $\mathcal{S}^j$ with modest additional computational costs---see Section~\ref{hierch}. 

We explain next an equivalent interpretation of the Schur complement elimination algorithm in terms of subdomain RtR map merging. In particular, the merging procedure will clarify the nature of the matrices in equations~\eqref{merge1} and~\eqref{merge2}.
  
  \subsection{Subdomain RtR map merging}\label{RtRM}
  
  We explain next in more detail the equivalence between (a) the Schur complement elimination of the unknowns $g_{j\ell}$ and $g_{\ell j}$ from the linear system~\eqref{system_A} and (b) an algebraic merging of the RtR maps $\mathcal{S}^j$ and $\mathcal{S}^{\ell}$ of the two adjacent subdomains $B_j$ and $B_\ell$ that delivers the RtR map of the box $B_{j\cup\ell}:=B_j\cup B_{\ell}$ containing in its interior the union of scatterers from $B_j$ and $B_{\ell}$. To this end, we start by defining the counterpart of the splitting in equations~\eqref{split1} for the subdomain $B_\ell$:
  \begin{equation}\label{split2}
\begin{bmatrix}\mathcal{S}^\ell_{j\ell,\ell j}&\mathcal{S}^\ell_{j\ell,\ell\ell'} & \mathcal{S}^{\ell}_{j\ell,\ell\hat{\ell}}\\ \mathcal{S}^\ell_{\ell'\ell,\ell j}& \mathcal{S}^\ell_{\ell'\ell,\ell\ell'} & \mathcal{S}^\ell_{\ell'\ell,\ell\hat{\ell}} \\ \mathcal{S}^\ell_{\hat{\ell}\ell,\ell j}& \mathcal{S}^\ell_{\hat{\ell}\ell,\ell\ell'} & \mathcal{S}^\ell_{\hat{\ell}\ell,\ell\hat{\ell}} \end{bmatrix}\begin{bmatrix}g_{\ell j}\\g_{\ell \ell'}\\g_{\ell\hat{\ell}}\end{bmatrix}=\begin{bmatrix}(\partial_{n_\ell}u_\ell+i\eta u_\ell)|_{\partial B_\ell\cap \partial B_j}\\(\partial_{n_\ell}u_\ell+i\eta u_\ell)|_{\partial B_\ell\cap \partial B_{\ell'}}\\(\partial_{n_\ell}u_\ell+i\eta u_\ell)|_{\partial B_{\hat{\ell}}}\end{bmatrix}
\end{equation}
  where we denoted $\partial B_{\hat{\ell}}:=\partial B_\ell\setminus(\partial B_j\cup\partial B_{\ell'})$. The equations corresponding to using the first rows of the matrices in formulas~\eqref{split1} and~\eqref{split2} amount to
  \begin{eqnarray*}
    \mathcal{S}^j_{\ell j,j\ell}g_{j\ell} +\mathcal{S}^j_{\ell j,jj'}g_{jj'}+\mathcal{S}^j_{\ell j,j\hat{j}}g_{j\hat{j}}&=&(\partial_{n_j}u_j+i\eta u_j)|_{\partial B_j\cap\partial B_\ell}\\
    \mathcal{S}^\ell_{j\ell, \ell j}g_{\ell j} +\mathcal{S}^\ell_{j\ell,\ell\ell'}g_{\ell\ell'}+\mathcal{S}^\ell_{j\ell,\ell\hat{\ell}}g_{\ell\hat{\ell}}&=&(\partial_{n_j}u_j+i\eta u_j)|_{\partial B_\ell\cap\partial B_j}.
  \end{eqnarray*}
Using the fact that $n_j=-n_\ell$ on $\partial B_j\cap\partial B_\ell$ the Robin data matching on $\partial B_j\cap\partial B_\ell$ implies that
\begin{eqnarray*}
(\partial_{n_j} u_j + i\eta u_j)|_{\partial B_j\cap\partial B_\ell}&=&-(\partial_{n_\ell} u_{\ell} -i\eta u_{\ell})|_{\partial B_j\cap\partial B_\ell}=-g_{\ell j},\nonumber\\
(\partial_{n_\ell} u_\ell +i\eta u_\ell)|_{\partial B_j\cap\partial B_\ell}&=&-(\partial_{n_j} u_j -i\eta u_j)|_{\partial B_j\cap\partial B_\ell}=-g_{j\ell}\nonumber
\end{eqnarray*}
and hence we obtain
\begin{equation*}
  \mathcal{D}_{j\ell}\begin{bmatrix}g_{j\ell}\\g_{\ell j}\end{bmatrix}=-\begin{bmatrix}0 & \mathcal{S}^\ell_{j\ell,\ell\ell'}\\\mathcal{S}^j_{\ell j,jj'}& 0\end{bmatrix}\begin{bmatrix}g_{jj'}\\g_{\ell\ell'}\end{bmatrix}-\begin{bmatrix}0 & \mathcal{S}^\ell_{j\ell,\ell\hat{\ell}}\\\mathcal{S}^j_{\ell j,j\hat{j}}& 0\end{bmatrix}\begin{bmatrix}g_{j\hat{j}}\\g_{\ell\hat{\ell}}\end{bmatrix}
  \end{equation*}
from which it follows that
\begin{equation}\label{elimination_gjl}
\begin{bmatrix}g_{j\ell}\\g_{\ell j}\end{bmatrix}=-\mathcal{D}_{j\ell}^{-1}\begin{bmatrix}0 & \mathcal{S}^\ell_{j\ell,\ell\ell'}\\\mathcal{S}^j_{\ell j,jj'}& 0\end{bmatrix}\begin{bmatrix}g_{jj'}\\g_{\ell\ell'}\end{bmatrix}-\mathcal{D}_{j\ell}^{-1}\begin{bmatrix}0 & \mathcal{S}^\ell_{j\ell,\ell\hat{\ell}}\\\mathcal{S}^j_{\ell j,j\hat{j}}& 0\end{bmatrix}\begin{bmatrix}g_{j\hat{j}}\\g_{\ell\hat{\ell}}\end{bmatrix}.
\end{equation}
Inserting the newly found formula~\eqref{elimination_gjl} in the remaining two row equations in formulas~\eqref{split1} and~\eqref{split2} results in a relation between the Robin data $g_{j\hat{j}}, g_{\ell\hat{\ell}},g_{jj'},g_{\ell\ell'}$ and the quantities $((\partial_{n_j}u_j+i\eta u_j)|_{\partial B_{\hat{j}}},(\partial_{n_\ell}u_\ell+i\eta u_\ell)|_{\partial B_{\hat{\ell}}},(\partial_{n_j}u_j+i\eta u_j)|_{\partial B_j\cap\partial B_j'},(\partial_{n_\ell}u_\ell+i\eta u_\ell)|_{\partial B_\ell\cap\partial B_{\ell'}})$ respectively. Given that subdomain RtR maps are well defined (see Section~\ref{well-posed}), the latter relationship is in effect a block decomposition of the RtR map corresponding to the subdomain $B_{j\cup\ell}$ containing in its interior the union of scatterers from $B_j$ and $B_{\ell}$. We emphasize that the RtR map corresponding to the subdomain $B_{j\cup\ell}$ was derived from merging the $B_j$ subdomain RtR map and the $B_\ell$ subdomain RtR map. We are interested in particular in deriving an explicit formula for the merged subdomain $B_{j\cup\ell}$ RtR map corresponding to the interface $(\partial B_j\cap \partial B_{j'})\cup (\partial B_\ell\cap\partial B_{\ell'})$. To that end we make use of the equations that use the second rows in formulas~\eqref{split1} and~\eqref{split2}:
\begin{eqnarray*}
    \mathcal{S}^j_{j'j,j\ell}g_{j\ell} +\mathcal{S}^j_{j'j,jj'}g_{jj'}+\mathcal{S}^j_{j'j,j\hat{j}}g_{j\hat{j}}&=&(\partial_{n_j}u_j+i\eta u_j)|_{\partial B_j\cap\partial B_{j'}}\\
    \mathcal{S}^\ell_{\ell'\ell, \ell j}g_{\ell j} +\mathcal{S}^\ell_{\ell'\ell,\ell\ell'}g_{\ell\ell'}+\mathcal{S}^\ell_{\ell'\ell,\ell\hat{\ell}}g_{\ell\hat{\ell}}&=&(\partial_{n_j}u_j+i\eta u_j)|_{\partial B_\ell\cap\partial B_{\ell'}}.
\end{eqnarray*}
We insert formula~\eqref{elimination_gjl} in the relation above and we find
\begin{eqnarray}\label{merge_jljplp}
  &&\left(\begin{bmatrix}\mathcal{S}^j_{j'j,jj'}&0\\0&\mathcal{S}^\ell_{\ell'\ell,\ell\ell'}\end{bmatrix}-\begin{bmatrix} \mathcal{S}^{j}_{j' j, j\ell} & 0 \\ 0 & \mathcal{S}^{\ell}_{\ell' \ell, \ell j} \end{bmatrix}\mathcal{D}_{j\ell}^{-1}\begin{bmatrix}0 & \mathcal{S}^{\ell}_{j\ell,\ell\ell'} \\ \mathcal{S}^{j}_{\ell j, jj'} & 0\end{bmatrix}\right)\begin{bmatrix}g_{jj'}\\g_{\ell\ell'}\end{bmatrix}\nonumber\\
    &+&\left(\begin{bmatrix}\mathcal{S}^j_{j'j,j\hat{j}}&0\\0&\mathcal{S}^\ell_{\ell'\ell,\ell\hat{\ell}}\end{bmatrix}-\begin{bmatrix} \mathcal{S}^{j}_{j' j, j\ell} & 0 \\ 0 & \mathcal{S}^{\ell}_{\ell' \ell, \ell j} \end{bmatrix}\mathcal{D}_{j\ell}^{-1}\begin{bmatrix}0 & \mathcal{S}^{\ell}_{j\ell,\ell\hat{\ell}} \\ \mathcal{S}^{j}_{\ell j, j\hat{j}} & 0\end{bmatrix}\right)\begin{bmatrix}g_{j\hat{j}}\\g_{\ell\hat{\ell}}\end{bmatrix}\\
      &=&\begin{bmatrix}(\partial_{n_j}u_j+i\eta u_j)|_{\partial B_j\cap\partial B_{j'}}\\(\partial_{n_j}u_j+i\eta u_j)|_{\partial B_\ell\cap\partial B_{\ell'}}\end{bmatrix}\nonumber.
\end{eqnarray}
Clearly, the matrix multiplying the Robin data $(g_{jj'},g_{\ell\ell'})$ in equation~\eqref{merge_jljplp} coincides with the matrix $\mathcal{S}^1_{j\ell j'\ell'}$ defined in equation~\eqref{merge1}. Furthermore, the matrix $\mathcal{S}^1_{j\ell j'\ell'}$ can be construed as the restriction on $(\partial B_j\cap \partial B_{j'})\cup (\partial B_\ell\cap\partial B_{\ell'})$ of the subdomain $B_{j\cup\ell}$ RtR map corresponding to the interface $(\partial B_j\cap \partial B_{j'})\cup (\partial B_\ell\cap\partial B_{\ell'})$. By the same token, the matrix $\mathcal{S}^2_{j\ell j'\ell'}$ defined in equation~\eqref{merge2} can be construed as the restriction on $(\partial B_j\cap \partial B_{j'})\cup (\partial B_\ell\cap\partial B_{\ell'})$ of the subdomain $B_{j'\cup\ell'}$ RtR map corresponding to the interface $(\partial B_j\cap \partial B_{j'})\cup (\partial B_\ell\cap\partial B_{\ell'})$. Thus, the application of the Schur complement of the upper left corner submatrix of the matrix featured in equation~\eqref{Schur1} can be viewed as a merging of the $B_{j\cup\ell}$ subdomain RtR map and the $B_{j'\cup\ell'}$ subdomain RtR map. 

The conclusion of the discussion above is that the Gaussian elimination/Schur complement procedure applied to the linear system~\eqref{system_A} can be recast into the equivalent framework of computing the RtR map on $\partial B_0$ corresponding to the Helmholtz equation in $B_0$ and the ensemble of scatterers $S_p, p=1,\ldots,P$ starting from subdomain RtR maps $\mathcal{S}^j$. Specifically, we define an interior RtR map in the domain $ B_0$ that takes into account the relevant boundary conditions on each boundary $\partial S_{p},\ p=1,\ldots,P$; we show in Section~\ref{well-posed} that the map $\mathcal{S}^{int}$ is well defined for all wavenumbers $k$. The latter map is defined as
\begin{equation}\label{Sint}
\mathcal{S}^{int}(\psi) := (\partial_{n_0} u+ i\eta u)|_{\partial {B_0}}
\end{equation}
where $u$ is a solution of the Helmholtz equation
\begin{eqnarray*}
\Delta u+ k^2 u &=& 0\ {\rm in}\  B_0\setminus\cup_{p=1}^P  S_{p}\quad u=0\ {\rm on}\ \partial S_{p},\ p=1,\ldots,P\\
\partial_{n_0} u -i\eta u&=&\psi\ {\rm on}\ \partial B_0.
 \end{eqnarray*}
The map $\mathcal{S}^{int}$ is computed by mergings of subdomain RtR maps $\mathcal{S}^j$ per the prescriptions above. The RtR operator merging procedure was used recently in~\cite{Gillman1} for the solution of volumetric scattering problems. At the same time we merge the operators $Y^j$ defined in equation~\eqref{eq:Yj} to compute the operator $Y^{int}$ that maps the Robin data $(\partial_{n_0} u- i\eta u)|_{\partial {B_0}}$ to the Cauchy data $(u_,\partial_nu)$ on the collection of scatterers included in $B_0$. We also define the exterior RtR map for the domain $B_0$ as
\begin{equation}\label{Sext}
  \mathcal{S}^{ext}(\varphi) := (\partial_{n_0} u_0+ i\eta u_0)|_{\partial {B_0}}
\end{equation}
where $u^0$ is the solution to the following Helmholtz equation in $\mathbb{R}^2\setminus B_0$ with Robin data $\varphi$ on $\partial B_0$:
\begin{eqnarray*}
  \Delta u_0 +k^2u_0&=&0\ {\rm in}\ \mathbb{R}^2\setminus B_0\nonumber\\
  \partial_{n_0} u_0 -i\eta u_0&=&\varphi\ {\rm on}\ \partial B_0\\
  \lim_{|r|\to\infty}r^{1/2}(\partial u_0/\partial r - iku_0)&=&0
\end{eqnarray*}
and $\partial_{n_0}$ is the normal derivative on $\partial B_0$ with respect to the unit normal exterior $n_0$ to $ B_0$. We note that the solution $u_0$ of the Robin boundary value problem described above is unique as long as $\eta<0$~\cite{KressColton} for all positive wavenumbers $k$ and data $\varphi\in L^2(\partial \Omega_0)$. The last stage of our algorithm consists of solving the reduced system~\eqref{system_A_reduced}. In the language of RtR maps, this last stage consists of using the relations
\begin{eqnarray}\label{last_merge}
  \mathcal{S}^{int}(\partial_{n_0} u^s -i\eta\ u^s)|_{\partial  B_0}+\mathcal{S}^{int}(\partial_{n_0} u^{inc} -i\eta\ u^{inc})|_{\partial  B_0}&=&(\partial_{n_0} u^s +i\eta\ u^s)|_{\partial  B_0}\nonumber\\
  &+&(\partial_{n_0} u^{inc} +i\eta\ u^{inc})|_{\partial  B_0}\\
  \mathcal{S}^{ext}(\partial_{n_0} u^s -i\eta\ u^s)|_{\partial  B_0}&=&(\partial_{n_0} u^s +i\eta\ u^s)|_{\partial  B_0}
\end{eqnarray}
to derive the following equation for the Robin data on $\partial B_0$:
\begin{equation}\label{Robin_data_box}
  (\mathcal{S}^{ext}-\mathcal{S}^{int})(\partial_{n_0} u^s -i\eta\ u^s)|_{\partial  B_0}=\mathcal{S}^{int}(\partial_{n_0} u^{inc} -i\eta\ u^{inc})|_{\partial  B_0}-(\partial_{n_0} u^s +i\eta\ u^s)|_{\partial  B_0}.
  \end{equation}
The solution $u^s$ of the scattering problem~\eqref{scatt} can be then retrieved both in the exterior of the box $ B_0$ (and hence in the far field) and the interior of the box $B_0$ from knowledge of $u^s|_{\partial B_0}$ and $\partial_n u^s|_{\partial B_0}$, which, in turn, can be computed through the following sequence: 
 \begin{enumerate}\label{merge}
  \item $(\partial_{n_0} u^s -i\eta\ u^s)|_{\partial  B_0}=(\mathcal{S}^{ext}-\mathcal{S}^{int})^{-1}\left(\mathcal{S}^{int}(\partial_{n_0} u^{inc} -i\eta\ u^{inc})|_{\partial  B_0}-(\partial_{n_0} u^{inc} +i\eta\ u^{inc})|_{\partial  B_0}\right)$
  \item  $(\partial_{n_0} u^s + i\eta\ u^s)|_{\partial  B_0}=\mathcal{S}^{ext}(\partial_{n_0} u^s - i\eta\ u^s)|_{\partial  B_0}$
  \item $u^s|_{\partial B_0} = \frac{1}{2i\eta}\left[(\partial_{n_0} u^s + i\eta\ u^s)|_{\partial  B_0} - (\partial_{n_0} u^s - i\eta\ u^s)|_{\partial  B_0}\right]$
  \item $\partial_{n_0} u^s|_{\partial B_0} = \frac{1}{2}\left[(\partial_{n_0} u^s + i\eta\ u^s)|_{\partial  B_0} + (\partial_{n_0} u^s - i\eta\ u^s)|_{\partial  B_0}\right]$.
 \end{enumerate}
 In order to carry out Step 1 in the four step program above we pursue the following approach:
 \begin{itemize}
 \item Compute each of the RtR maps for the subdomains $ B_j, j=1,\ldots,L$ via well-conditioned boundary integral equations and then use the merging procedure outlined above to compute $\mathcal{S}^{int}$. The merging procedure is performed in a hierarchical manner that optimizes the computational cost of this stage 
 \item Compute $\mathcal{S}^{ext}$ using well conditioned boundary integral equations
   \item Solve for the quantity $(\partial_{n_0} u^s -i\eta\ u^s)|_{\partial  B_0}$ from equation~\eqref{Robin_data_box}.
 \end{itemize}
 The validity of the Gaussian elimination algorithm/RtR map merging described above hinges on two important questions: (I) the fact that the RtR maps $\mathcal{S}^j$ and $\mathcal{S}^{int}$ are well defined for all real wavenumbers $k$, and (II) the validity of equations~\eqref{inv_matrix_explicit}. We start by establishing the fact that the subdomain RtR maps $\mathcal{S}^j$ are well defined for all real wavenumbers $k$.
 
\subsection{Well posedness of the subdomain Robin problems\label{well-posed}}

Before we establish the main results about the well posedness of Helmholtz equation with Robin boundary conditions we briefly review the definition of Sobolev spaces in Lipschitz domains. For any $D\subset\mathbb{R}^2$ domain with bounded Lipschitz boundary $\Gamma$,  we denote by $H^s(D)$ the classical Sobolev space of order $s$ on $D$ 
(see for example in the monographs \cite[Ch. 3]{mclean:2000} or \cite[Ch. 2]{adams:2003}). 
We consider in addition the Sobolev spaces defined on the boundary $\Gamma$,  $H^s(\Gamma)$, which are well defined for any $s\in[-1,1]$. We recall that for any $s>t$, $H^s(\Sigma)\subset H^t(\Sigma)$, $\Sigma\in\{D_1,D_2,\Gamma\}$ with compact support. Moreover, and
$\big(H^t(\Gamma)\big)'=H^{-t}(\Gamma)$ when the inner product of $H^0(\Gamma)=L^2(\Gamma)$ is used as duality product. Let $\Gamma_0\subset\Gamma$ such that $meas(\Gamma_0)>0$. For $0<s\leq 1/2$ we define by ${H}^s(\Gamma_0)$ be the space of distributions that are restrictions to $\Gamma_0$ of functions in $H^s(\Gamma)$. The space $\widetilde{H}^s(\Gamma_0)$ is defined as the closed subspace of $H^s(\Gamma_0)$
\[
\widetilde{H}^s(\Gamma_0)=\{u\in H^s(\Gamma_0):\widetilde{u}\in H^s(\Gamma)\},\ 0<s\leq 1/2
\]
where
\[\widetilde{u}:=\begin{cases}
 u, & {\rm on}\  \Gamma \\
 0,  & {\rm on}\ \Gamma\setminus\Gamma_0.
\end{cases}
\]
We define then $H^t(\Gamma_0)$ to be the dual of $\widetilde{H}^{-t}(\Gamma_0)$ for $-1/2\leq t<0$, and $\widetilde{H}^t(\Gamma_0)$ the dual of $H^{-t}(\Gamma_0)$ for $-1/2\leq t<0$.

In order to keep the notations simple, we consider the case of one closed Lipschitz scatterer $S$ inside of a box subdomain $B$  and the following Helmholtz boundary value problem
\begin{equation} 
  \label{eq:Ac_i}
\begin{aligned}
  \Delta u+k^2 u&=&0,\qquad &\mathrm{in}\  B\setminus S\\
  u&=&0,\qquad & \mathrm{on}\ \partial S\\
\partial_n u - i\eta u&=&f,\qquad &\mathrm{on}\ \partial B
\end{aligned}
\end{equation}
where the wavenumber $k$ is assumed to be positive, $f$ is data defined on $\partial B$ and $f\in H^{-1/2}(\partial B)$, and $\eta$ is assumed to have the properties $\eta\in\mathbb{R},\ \pm\eta >0$. In equations~\eqref{eq:Ac_i} the normal derivative is taken with  respect to the unit normal pointing outside of the domain $B$. The first result we establish is:
\begin{theorem}\label{tm1}
For data $f\in H^{-1/2}(\partial B)$ the equations~\eqref{eq:Ac_i} has a unique solution $u\in H^1_{\Delta}( B\setminus S):=\{U\in H^1( B\setminus S)\ :\ \Delta U \in L^2( B\setminus S)\}$.
\end{theorem}
\begin{proof}
  We settle here the issue of uniqueness. For existence results we refer to the proof of Theorem~\ref{thm2}. In order to establish uniqueness of solutions, we show that if $f=0$, then a function $u$ that satisfies equations~\eqref{eq:Ac_i} must be identically zero in the domain $ B\setminus S$. We have that
  \[
  \int_{ B\setminus S}(|\nabla u|^2 -k^2|u|^2)dx=-i\eta\int_{\partial B}|u|^2ds,
  \]
  from which it follows that $u|_{\partial B}=0$. Given that $f=0$ on $\partial B$, the last fact implies in turn that $\partial_n u|_{\partial B}=0$. Let $\mathbf{x}_0\in\partial B$ that is not a corner point, and choose $\varepsilon$ small enough so that $B_\varepsilon(\mathbf{x}_0)=\{\mathbf{y}\in \mathbb{R}^2: |\mathbf{y}-\mathbf{x_0}|<\epsilon\}$ does not contain a corner point of $\partial B$. Denote by
  \[
  v(\mathbf{x}):=\begin{cases}
 u(\mathbf{x}), & \mathbf{x}\in B_\varepsilon(\mathbf{x}_0)\cap\overline{B} \\
 0, & \mathbf{x}\in B_\varepsilon(\mathbf{x}_0)\setminus \overline{B}.
\end{cases}
  \]
  Given that both $u$ and $\partial_nu$ vanish on $B_\varepsilon(\mathbf{x}_0)\cap\partial B$, it follows that $v$ is weak solution of the Helmholtz equation in $B_\varepsilon(\mathbf{x}_0)$, and thus it is a strong solution. Since $v$ is identically zero in an open set, analyticity arguments imply that $v$ is zero everywhere, so in particular $u$ is identically zero in open set. The latter implies that $u$ is zero in $B\setminus S$.
\end{proof}

In the case when $B$ are convex domains, standard interior elliptic estimates imply that the solution $u$ of equations~\eqref{eq:Ac_i} has improved regularity in a neighborhood of $\partial B$, that is $u\in H^{2}(\partial B(\delta))$ with $\partial B(\delta)=\{\mathbf{x}\in B: dist(\mathbf{x},\partial B)<\delta\}$ for small enough $\delta$. This improved regularity implies that $u|_{\partial B}\in H^1(\partial B)$ and $\partial_n u|_{\partial B}\in L^2(\partial B)$. Thus, it makes sense to look at problems~\eqref{eq:Ac_i} with Robin data $f\in L^2(\partial B)$, so that $(\partial_n u +i\eta u)|_{\partial B}\in L^2(\partial B)$. As previously mentioned, a central role in our DDM method is played by the RtR operator $\mathcal{S}:L^2(\partial B)\to L^2(\partial B)$ defined as
\begin{equation}\label{def_RtR}
  \mathcal{S}(f):=\partial_n u +  i\eta u
\end{equation}
where $u$ is the solution of  equations~\eqref{eq:Ac_i}. The operator $\mathcal{S}$ can be easily seen to be unitary in $L^2(\partial B)$, a property that is essential in establishing the convergence of fixed point iterative DD methods~\cite{Depres}. Having proved these basic facts about the RtR maps $\mathcal{S}^j$, we investigate next the validity of the Schur complement elimination procedure.



\subsection{Merging of RtR maps theoretical considerations\label{hierch}}

 The central issue in the Schur complement/RtR merger procedure is the validity of formula~\eqref{inv_matrix_explicit}. We investigate this problem in the representatitve case of a left/right merging of RtR maps for two subdomains (boxes) arranged as in Figure~\ref{fig:domain2} so that $B_j$ is the subdomain on the left containing the scatterer $S_j$ and $B_{\ell}$ is the subdomain on the right containing the scatterer $S_\ell$. The top down merging is amenable to a similar treatment. We denote $\partial B_j\setminus\partial B_{\ell}$ by $L$, $\partial B_{\ell}\setminus \partial B_j$ by $R$, and the common edge $\partial B_j\cap \partial B_{\ell}$ by $C$. The maps $\mathcal{S}^j$ and $\mathcal{S}^{\ell}$ were expressed in block form in the following manner
\[
\mathcal{S}^{j}=\begin{bmatrix}\mathcal{S}^j_{LL}&\mathcal{S}^j_{LC}\\ \mathcal{S}^j_{CL}& \mathcal{S}^j_{CC}\end{bmatrix}\qquad \mathcal{S}^{\ell}=\begin{bmatrix}\mathcal{S}^{\ell}_{RR}&\mathcal{S}^{\ell}_{RC}\\ \mathcal{S}^{\ell}_{CR}& \mathcal{S}^{\ell}_{CC}\end{bmatrix}
\]
where the block operators are defined informally as
\begin{equation}\label{SL}
\begin{bmatrix}\mathcal{S}^j_{LL}&\mathcal{S}^j_{LC}\\ \mathcal{S}^j_{CL}& \mathcal{S}^j_{CC}\end{bmatrix}\begin{bmatrix}(\partial_{n_j} u_j -i\eta u_j)|_{L}\\(\partial_{n_j} u_j -i\eta u_j)|_{C}\end{bmatrix}=\begin{bmatrix}(\partial_{n_j} u_j + i\eta u_j)|_{L}\\(\partial_{n_j} u_j + i\eta u_j)|_{C}\end{bmatrix}
\end{equation}
and
\begin{equation}\label{SR}
\begin{bmatrix}\mathcal{S}^{\ell}_{RR}&\mathcal{S}^{\ell}_{RC}\\ \mathcal{S}^{\ell}_{CR}& \mathcal{S}^{\ell}_{CC}\end{bmatrix}\begin{bmatrix}(\partial_{n_\ell} u_{\ell} -i\eta u_{\ell})|_{R}\\(\partial_{n_\ell} u_{\ell} -i\eta u_{\ell})|_{C}\end{bmatrix}=\begin{bmatrix}(\partial_{n_\ell} u_{\ell} + i\eta u_{\ell})|_{R}\\(\partial_{n_\ell} u_{\ell} + i\eta u_{\ell})|_{C}\end{bmatrix}.
\end{equation}
A more precise definition of the block operators in equations~\eqref{SL} and~\eqref{SR} can be given by considering the partial RtR maps:
\[
\mathcal{S}^{j}_C(0,\varphi_C) := (\partial_{n_j} w_j+ i\eta w_j)|_{\partial B_j},
\]
where
\begin{eqnarray}\label{Robin_int}
  \Delta w_j+k^2 w_j&=&0\ {\rm in}\  B_j\setminus S_{j},\quad w_j=0\ {\rm on}\ \partial S_{j} \nonumber\\
  \partial_{n_j} w_j-i\eta w_j&=&\varphi_C\ {\rm on}\quad C,\\
  \partial_{n_j} w_j-i\eta w_j&=&0\ {\rm on}\quad L.\nonumber
\end{eqnarray}
In equations~\eqref{Robin_int} the data $\varphi_C$ is such that $\varphi_C\in\widetilde{H}^{-1/2}(C)$, which implies that $(0,\varphi_C)\in H^{-1/2}(\partial B_j)$. We use the restriction operators $R_C^j:H^{-1/2}(\partial B_j)\to H^{-1/2}(C)$ defined via duality pairings in the form $\langle R_{C}^j\psi,\varphi\rangle=\langle f,E_{C}^j\varphi\rangle$, where $\psi\in H^{-1/2}(\partial B_j)$, $\varphi\in \widetilde{H}^{1/2}(C)$, and $E_{C}^j:\widetilde{H}^{-1/2}(C)\to H^{1/2}(\partial B_j)$ is the extension by zero operator. Then the operators $\mathcal{S}^j_{CC}$ are simply defined as $\mathcal{S}^j_{CC}:=R_C^j\mathcal{S}^{j}_C$, so that $\mathcal{S}^j_{CC}:\widetilde{H}^{-1/2}(C)\to H^{-1/2}(C)$. The operators $\mathcal{S}^j_{LC}$ are then similarly defined.

Using the same procedure we define the operators
\[
\mathcal{S}^{j}_L(\varphi_L,0) := (\partial_{n_j} v_j+ i\eta v_j)|_{\partial B_j},
\]
where
\begin{eqnarray*}
  \Delta v_j+k^2 v_j&=&0\ {\rm in}\  B_j\setminus S_{j},\quad v_j=0\ {\rm on}\ \partial S_{j}\\
  \partial_{n_j}v_j-i\eta v_j &=&\varphi_L\ {\rm on}\quad L\\
  \partial_{n_j} v_j-i\eta v_j&=&0\ {\rm on}\quad C.
\end{eqnarray*}
Denoting by $R_L^j$ the restriction operator $R_L^j:H^{-1/2}(\partial B_j)\to H^{-1/2}(L)$, the operators $\mathcal{S}^j_{LL}$ and $\mathcal{S}^j_{CL}$ are simply defined as $\mathcal{S}^j_{LL}=R_L^j\mathcal{S}^{j}_L$ and $\mathcal{S}^j_{CL}=R_C^j\mathcal{S}^{j}_L$. The operators $\mathcal{S}^{\ell}_{RR}$, $\mathcal{S}^{\ell}_{RC}$, $\mathcal{S}^{\ell}_{CR}$, and $\mathcal{S}^{\ell}_{CC}$ are defined similarly.

\begin{figure}
\centering
\includegraphics[width=60mm]{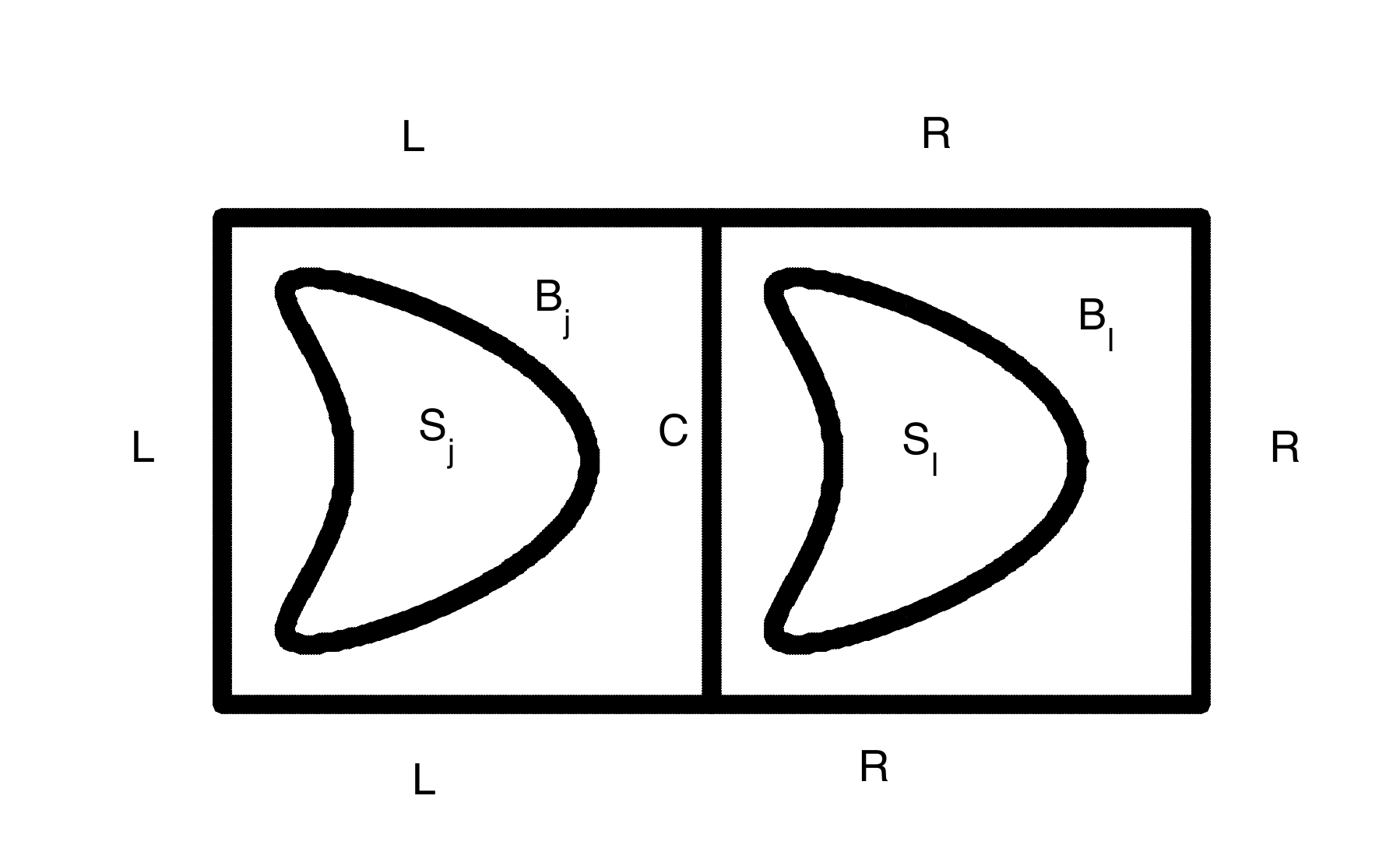}\includegraphics[width=60mm]{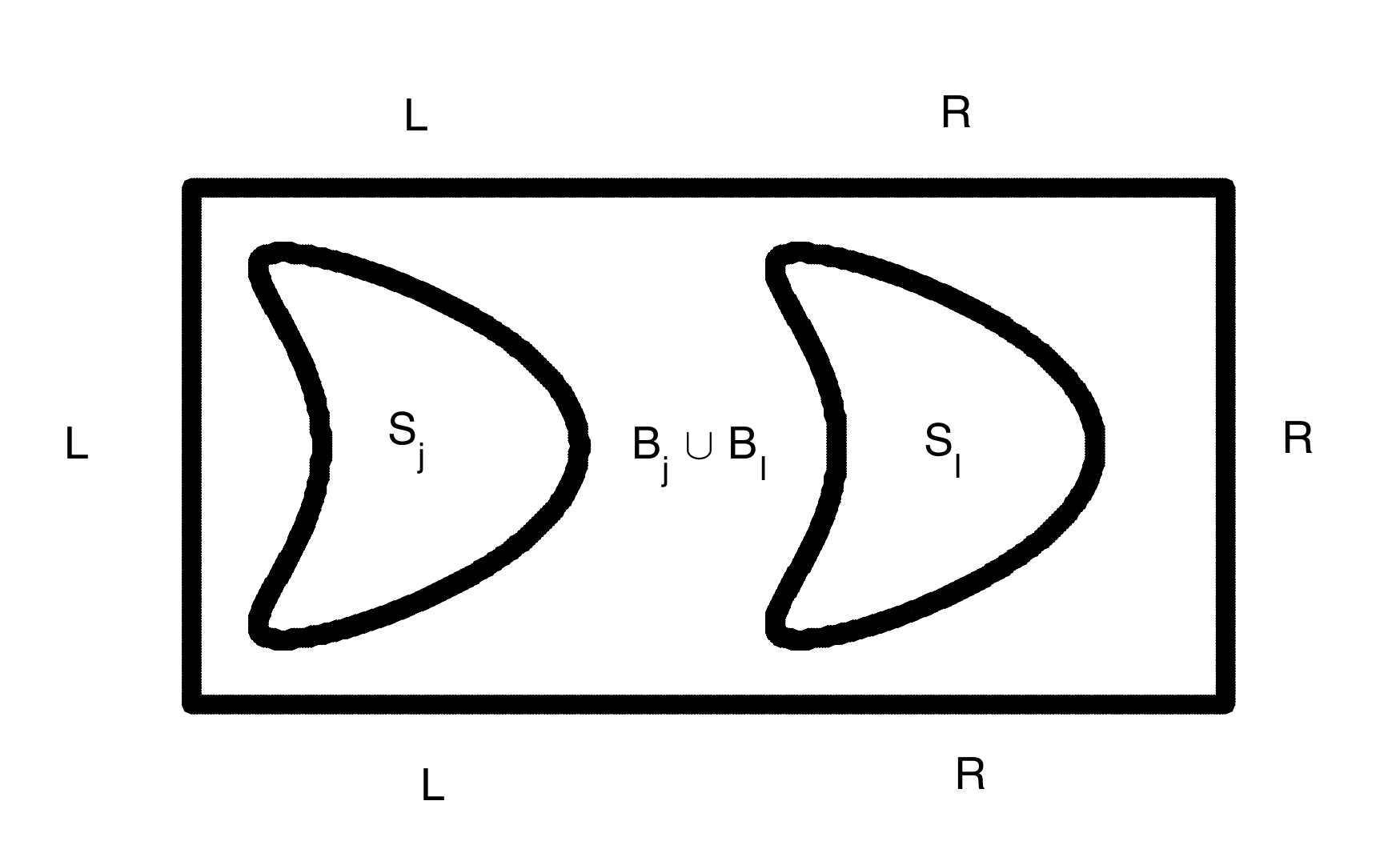}
\caption{Typical box domain.}
\label{fig:domain2}
\end{figure}

Applying the procedure of eliminating the Robin data on the common interface $C$ from equations~\eqref{SL} and~\eqref{SR} we derive the merged expression for the RtR operator for the domain $ B_j\cup B_{\ell}$ which is akin to that in formula~\eqref{merge_jljplp}:
\begin{equation}\label{eq:merged}
  \mathcal{S}^{j\cup \ell}\begin{bmatrix}(\partial_n u_{j,\ell} -i\eta u_{j,\ell})|_{L}\\(\partial_n u_{j,\ell} -i\eta u_{j,\ell})|_{R}\end{bmatrix}=\begin{bmatrix}(\partial_n u_{j,\ell} + i\eta u_{j,\ell})|_{L}\\(\partial_n u_{j,\ell} + i\eta u_{j,\ell})|_{R}\end{bmatrix}
\end{equation}
where
\[
\mathcal{S}^{c}=(\mathcal{S}^{\ell}_{CC}\mathcal{S}^j_{CC}-I)^{-1}
\]
and
\[
\mathcal{S}^{j\cup \ell}=\begin{bmatrix}\mathcal{S}^{j}_{LL}-\mathcal{S}^{j}_{LC}\mathcal{S}^{c}\mathcal{S}^{\ell}_{CC}\mathcal{S}^{j}_{CL} & \mathcal{S}^{j}_{LC}\mathcal{S}^{c} \mathcal{S}^{\ell}_{CR}\\ -\mathcal{S}^{\ell}_{RC}\mathcal{S}^{j}_{CL}+\mathcal{S}^{\ell}_{RC}\mathcal{S}^j_{CC}\mathcal{S}^{c}\mathcal{S}^{\ell}_{CC}\mathcal{S}^j_{CL} & \mathcal{S}^{\ell}_{RR}-\mathcal{S}^{\ell}_{RC}\mathcal{S}^j_{CC}\mathcal{S}^{c}\mathcal{S}^{\ell}_{CR}\end{bmatrix}.
\]
\begin{remark}
  Formulas~\eqref{eq:merged} also appear in~\cite{Gillman1}. The merging procedure above also delivers a merged map of Robin data on the boundary of $B_j\cup B_{\ell}$ to Neumann traces on the scatterers inside of $B_j\cup B_{\ell}$. Indeed, splitting the maps $Y^j[\partial B_j, \partial S_j]$ defined in equations~\eqref{Y_explicit} so that to account for the left (L) and common (C) contributions, we get
  \begin{equation}\label{mergedY}
    Y^{j\cup \ell}[\partial(B_j\cup B_{\ell}),\partial S_j]=\begin{bmatrix}Y^j_{L}-Y^j_C\mathcal{S}^c\mathcal{S}^{\ell}_{CC}\mathcal{S}^j_{CL}\\
Y^j_C\mathcal{S}^c\mathcal{S}^{\ell}_{CR}\end{bmatrix}.
  \end{equation}
  Similar equations can be derived for the merged maps $Y^{j\cup \ell}[\partial(B_j\cup B_{\ell}),\partial S_{\ell}]$. We note that the merged maps $\mathcal{S}^{j\cup \ell}$ and $Y^{j\cup \ell}$ allow us to retrieve the values of the fields $u^j$ and $u^{\ell}$ everywhere in the interior of the box $B^j\cup B^{\ell}$ from knowledge of Robin data on the boundary of $B^j\cup B^{\ell}$.
  \end{remark}
Clearly, the central issue in the merging procedure above is the invertibility of the operator $\mathcal{S}^{\ell}_{CC}\mathcal{S}^j_{CC}-I$, which we establish in Theorem~\ref{thm_inv}. We begin with a result that sheds light into the spectral properties of the operators $\mathcal{S}^j_{CC}$ and $\mathcal{S}^\ell_{CC}$:
\begin{theorem}\label{spectral_radius}
  The operators $\mathcal{S}^j_{CC}$ and $\mathcal{S}^\ell_{CC}$ can be expressed as $I+\mathcal{K}^j$ and $I+\mathcal{K}^\ell$ respectively, where $\mathcal{K}^j,\ \mathcal{K}^\ell:\widetilde{H}^{-1/2}(C)\to H^{1/2}(C)$ continuously.
\end{theorem}
\begin{proof} Clearly we have that
  \[
  \mathcal{S}^j_{CC}\varphi_C=\varphi_C+2i\eta\ w_j\qquad {\rm on}\quad C
  \]
  where $w_j$ is the solution of equations~\eqref{Robin_int}. Since $w_j\in H^1(B\setminus S)$ it follows that $w_j\in H^{1/2}(\partial B)$ and hence $R_C^jw_j\in H^{1/2}(C)$. 
\end{proof}

An immediate consequence of the result established in Theorem~\ref{spectral_radius} is that the operator $\mathcal{S}_{CC}^\ell\mathcal{S}_{CC}^j-I:\widetilde{H}^{-1/2}(C)\to H^{1/2}(C)$ continuously. In order to establish the invertibility of this operator we make use of the four boundary integral operators associated with the Calderon calculus for a Lipschitz domain. Let $D$ be a bounded domain in $\mathbb{R}^2$ whose boundary $\Gamma$ is a closed Lipschitz curve. Given a wavenumber $k$ such that $\Re{k}>0$ and $\Im{k}\geq 0$, and a density $\varphi$ defined on $\Gamma$, we define the single layer potential as
$$[SL_{\Gamma,k}(\varphi)](\mathbf{z}):=\int_\Gamma G_k(\mathbf{z}-\mathbf{y})\varphi(\mathbf{y})ds(\mathbf{y}),\ \mathbf{z}\in\mathbb{R}^2\setminus\Gamma$$
and the double layer potential as
$$[DL_{\Gamma,k}(\varphi)](\mathbf{z}):=\int_\Gamma \frac{\partial G_k(\mathbf{z}-\mathbf{y})}{\partial\mathbf{n}(\mathbf{y})}\varphi(\mathbf{y})ds(\mathbf{y}),\ \mathbf{z}\in\mathbb{R}^2\setminus\Gamma$$
where $G_k(\mathbf{x})=\frac{i}{4}H_0^{(1)}(k|\mathbf{x}|)$ represents the two-dimensional Green's function of the Helmholtz equation with wavenumber $k$. Applying Dirichlet and Neumann exterior and interior traces on $\Gamma$ (denoted by $\gamma_\Gamma^{D,1}$ and $\gamma_\Gamma^{D,2}$ and respectively $\gamma_\Gamma^{N,1}$ and $\gamma_\Gamma^{N,2}$) to the single and double layer potentials corresponding to the wavenumber $k$ and a density $\varphi$ we define the four Helmholtz boundary integral operators
\begin{eqnarray}\label{traces}
\gamma_\Gamma^{D,1} SL_{\Gamma,k}(\varphi)&=&\gamma_\Gamma^{D,2} SL_{\Gamma,k}(\varphi)=S_{\Gamma,k}\varphi,\quad \gamma_\Gamma^{N,1} DL_{\Gamma,k}(\varphi)=\gamma_\Gamma^{N,2} DL_k(\varphi)=N_{\Gamma,k}\varphi\nonumber\\
\gamma_\Gamma^{N,j} SL_{\Gamma,k}(\varphi)&=&(-1)^j\frac{\varphi}{2}+K_{\Gamma,k}^\top \varphi\quad j=1,2,\quad
\gamma_\Gamma^{D,j} DL_{\Gamma,k}(\varphi)=(-1)^{j}\frac{\varphi}{2}+K_{\Gamma,k}\varphi\quad j=1,2.\nonumber
\end{eqnarray}

\subsubsection{Invertibility of the operator $\mathcal{S}_{CC}^\ell\mathcal{S}_{CC}^j-I$}\label{inv}
We are now ready to prove the main theoretical result that guarantees that the Schur complement procedure does not break down:
\begin{theorem}\label{thm_inv}
  The operator $\mathcal{S}_{CC}^\ell\mathcal{S}_{CC}^j-I:\widetilde{H}^{-1/2}(C)\to H^{1/2}(C)$ is injective and onto, and thus its inverse is continuous.
\end{theorem}
\begin{proof}
Let $\psi\in H^{1/2}(C)$ and consider the equation
\begin{equation}\label{eq:surj}
  \mathcal{S}_{CC}^\ell\mathcal{S}_{CC}^j\varphi-\varphi=\psi\quad {\rm for}\ \varphi\in\widetilde{H}^{-1/2}(C).
\end{equation}
  Let $v_j$ be the solution of the following Helmholtz equation
  \begin{eqnarray}\label{RobinL}
  \Delta v_j+k^2 v_j&=&0\ {\rm in}\  B_j\setminus S_j, \quad v_j=0\ {\rm on}\ \partial S_j \nonumber\\
  \partial_{n_j} v_j-i\eta\ v_j&=&\varphi\ {\rm on}\quad C,\\
  \partial_{n_j} v_j-i\eta\ v_j&=0&\ {\rm on}\quad L.\nonumber
  \end{eqnarray}
  and $v_\ell$ be the solution of the following Helmholtz equation
  \begin{eqnarray}\label{RobinR}
  \Delta v_\ell+k^2 v_\ell&=&0\ {\rm in}\  B_\ell\setminus S_\ell, \quad v_\ell=0\ {\rm on}\ \partial S_\ell\nonumber\\
  \partial_{n_\ell} v_\ell-i\eta\ v_\ell&=&\partial_{n_j}v_j+i\eta\ v_j\ {\rm on}\quad C,\\
  \partial_{n_\ell} v_\ell-i\eta\ v_\ell&=0&\ {\rm on}\quad R.\nonumber
  \end{eqnarray}
  Eliminating $\varphi$ from equations~\eqref{eq:surj},~\eqref{RobinL}, and~\eqref{RobinR} we obtain
  \begin{eqnarray*}
    \partial_{n_\ell} v_\ell+i\eta\ v_\ell&=&\partial_{n_j} v_j-i\eta\ v_j+\psi\ {\rm on}\quad C,\\
    \partial_{n_\ell} v_\ell-i\eta\ v_\ell&=&\partial_{n_j} v_j+i\eta\ v_j\ {\rm on}\quad C.
  \end{eqnarray*}
  Defining $\tilde{v}_\ell:=-v_\ell$ in $ B_\ell\setminus  S_\ell$, we see that the last two equations imply that
  \begin{eqnarray}\label{jump}
    \tilde{v}_\ell&=&v_j+\frac{i}{2\eta}\psi\ {\rm on}\quad C\label{jump1}\\
    -\partial_{n_\ell}\tilde{v}_\ell&=&\partial_{n_j}v_j+\frac{1}{2}\psi\ {\rm on}\quad C.\label{jump2}
  \end{eqnarray}
  We assume in what follows that the boundary conditions on the scatterers are of Dirichlet type. General types of boundary conditions can be treated similarly. We apply Green's identities in the domain $B_j\setminus S_{j}$ and obtain
  \begin{eqnarray}\label{Green1}
    v_j(\mathbf{x})&=&i\eta[SL_{L,k}(v_j)](\mathbf{x})-[DL_{L,k}(v_j)](\mathbf{x})+[SL_{C,k}(\partial_{n_j}v_j)](\mathbf{x})-[DL_{C,n_j,k}(v_j)](\mathbf{x})\nonumber\\
    &-&[SL_{S_j,k}\partial_n v_j](\mathbf{x}),\ \mathbf{x}\in B_j\setminus S_{j}.
  \end{eqnarray}
  On the other hand, applying the Green's identities in the domain $B_\ell\setminus S_{\ell}$ for the functions $\tilde{v}_\ell$ and $G_k(\mathbf{x}-\cdot)$ with $\mathbf{x}\in B_j\setminus S_{j}$ we obtain
 \begin{eqnarray}\label{Green2}
    0&=&i\eta[SL_{R,k}(\tilde{v}_\ell)](\mathbf{x})-[DL_{R,k}(\tilde{v}_\ell)](\mathbf{x})+[SL_{C,k}(\partial_{n_\ell}\tilde{v}_\ell)](\mathbf{x})-[DL_{C,n_\ell,k}(\tilde{v}_\ell)](\mathbf{x})\nonumber\\
    &-&[SL_{S_{\ell},k}\partial_n \tilde{v}_\ell](\mathbf{x}),\ \mathbf{x}\in B_j\setminus S_{j}.
 \end{eqnarray}
 We chose to include the normals in the definition of the double layer potentials on $C$ in order to emphasize the fact that those are different (opposite one another) in formulas~\eqref{Green1} and~\eqref{Green2}. We add equations~\eqref{Green1} and~\eqref{Green2}, and we take into account the relations~\eqref{jump} to obtain
 \begin{eqnarray}\label{GreenL}
   v_j(\mathbf{x})&=&i\eta[SL_{L,k}(v_j)](\mathbf{x})+i\eta[SL_{R,k}(\tilde{v}_\ell)](\mathbf{x})-[DL_{L,k}(v_j)](\mathbf{x})-[DL_{R,k}(\tilde{v}_\ell)](\mathbf{x})\nonumber\\
   &+&\frac{i}{2\eta}[DL_{C,n,k}(\psi)](\mathbf{x})-\frac{1}{2}[SL_{C,k}(\psi)](\mathbf{x})-[SL_{S_{j},k}\partial_n v_j](\mathbf{x})-[SL_{S_{\ell},k}\partial_n \tilde{v}_\ell](\mathbf{x}),\ \mathbf{x}\in B_j\setminus S_{j}\nonumber\\
 \end{eqnarray}
 where we used one more time the notation $n=n_j=-n_\ell$. A similar relation can be derived for $\tilde{v}_\ell$ in the domain  $B_\ell\setminus S_{\ell}$. If we define
\[
w:=\left\{ \begin{array}{cc} v_j &\quad{\rm in} \quad B_j\\ \tilde{v}_\ell &\quad {\rm in }\quad B_\ell\end{array}\right.
  \]
  the formula~\eqref{GreenL} and its $B_\ell$ counterpart can be expressed as
  \begin{eqnarray}\label{GreenLpR}
   w(\mathbf{x})&=&i\eta[SL_{L\cup R,k}(w)](\mathbf{x})-[DL_{L\cup R,k}(w)](\mathbf{x})+\frac{i}{2\eta}[DL_{C,n,k}(\psi)](\mathbf{x})-\frac{1}{2}[SL_{C,k}(\psi)](\mathbf{x})\nonumber\\
    &-&[SL_{S_{j}\cup S_\ell,k}\partial_n w](\mathbf{x}),\ \mathbf{x}\in (B_j\cup B_\ell)\setminus(S_j\cup S_\ell)\setminus C.
  \end{eqnarray}
  At this stage we apply to both sides of equation~\eqref{GreenLpR} (1) first the interior Dirichlet and Neumann traces on $\Gamma_{LR}:=L\cup R$ (note that $\Gamma_{LR}$ is a Lipschitz domain); (2) we combine the Dirichlet trace with the regularizing operator $2S_{\Gamma_{LR},\kappa},\ \Im(\kappa)>0$ applied to the Neumann trace; and then (3) the Dirichlet and Neumann traces on $S_{LR}:=S_j\cup S_\ell$ and we combine the latter in the typical Burton Miller fashion. Applying the interior Dirichlet trace on $\Gamma_{LR}$ to both sides of equation~\eqref{GreenLpR} leads to the relation
  \begin{equation}\label{DirLR}
    \left(\frac{1}{2}I+K_{\Gamma_{LR},k}-i\eta S_{\Gamma_{LR},k}\right)\gamma_{\Gamma_{LR}}^{D,2}w+SL_{S_{LR},\Gamma_{LR},k}\gamma_{S_{LR}}^{N,1}w|_{S_{LR}}=\left(\frac{1}{2}SL_{C,\Gamma_{LR},k}-\frac{i}{2\eta}DL_{C,\Gamma_{LR},k}\right)\psi.
  \end{equation}
  On the other hand, applying the interior Neumann trace on $\Gamma_{LR}$ to both sides of equation~\eqref{GreenLpR} while taking into account the fact that $\partial_nw+i\eta w=0$ on $\Gamma_{LR}$ leads to the relation
  \begin{eqnarray}\label{NeuLR}
    \left(\frac{i\eta}{2}I-i\eta K^\top_{\Gamma_{LR},k}+ N_{\Gamma_{LR},k}\right)\gamma_{\Gamma_{LR}}^{D,2}w&+&\gamma_{\Gamma_{LR}}^N SL_{S_{LR},\Gamma_{LR},k}\gamma_{S_{LR}}^{N,1}w|_{S_{LR}}\nonumber\\
    &=&\left[\gamma_{\Gamma_{LR}}^{N,2}\left(\frac{1}{2}SL_{C,\Gamma_{LR},k}-\frac{i}{2\eta}DL_{C,\Gamma_{LR},k}\right)\right]\psi.
  \end{eqnarray}
  Adding up the two sides of equation~\eqref{DirLR} and the two sides of equation~\eqref{NeuLR} multiplied on the left by $-2S_{\Gamma_{LR},\kappa},\ \Im(\kappa>0),\ \Re(\kappa)>0$ we obtain the following relation
  \begin{eqnarray}\label{eq:combined1}
    \mathcal{C}_{11}\gamma_{\Gamma_{LR}}^{D,2}w+\mathcal{C}_{12}\gamma_{S_{LR}}^{N,1}w|_{S_{LR}}&=&\mathcal{D}_1\psi\\
    \mathcal{C}_{11}&=&\frac{1}{2}I+K_{\Gamma_{LR},k}-i\eta S_{\Gamma_{LR},k}-i\eta S_{\Gamma_{LR},\kappa} +2i\eta S_{\Gamma_{LR},\kappa}K^\top_{\Gamma_{LR},k}\nonumber\\
    &-& 2S_{\Gamma_{LR},\kappa}N_{\Gamma_{LR},k}\nonumber\\
    \mathcal{C}_{12}&=&SL_{S_{LR},\Gamma_{LR},k}-2S_{\Gamma_{LR},\kappa}\gamma_{\Gamma_{LR}}^{N,2} SL_{S_{LR},\Gamma_{LR},k}\nonumber\\
    \mathcal{D}_1&=&\left(\frac{1}{2}SL_{C,\Gamma_{LR},k}-\frac{i}{2\eta}DL_{C,\Gamma_{LR},k}\right)\nonumber\\
    &-&2S_{\Gamma_{LR},\kappa}\left[\gamma_{\Gamma_{LR}}^{N,2}\left(\frac{1}{2}SL_{C,\Gamma_{LR},k}-\frac{i}{2\eta}DL_{C,\Gamma_{LR},k}\right)\right]\nonumber.
  \end{eqnarray}
  On the other hand, applying the exterior Neumann trace on $S_{LR}$ to equation~\eqref{GreenLpR} and combining it with $-i\mu,\ \mu\neq 0,\ \mu\in\mathbb{R}$ multiplied by the exterior Dirichlet trace  on $S_{LR}$ applied to the same equation we obtain a second relation of the form
 \begin{eqnarray}\label{eq:combined2}
    \mathcal{C}_{21}\gamma_{\Gamma_{LR}}^{D,2}w+\mathcal{C}_{22}\gamma_{S_{LR}}^{N,1}w|_{S_{LR}}&=&\mathcal{D}_2\psi\\
    -\mathcal{C}_{21}&=&i\eta \gamma_{S_{LR}}^{N,1}SL_{\Gamma_{LR}, S_{LR},k}-\gamma_{S_{LR}}^{N,1}DL_{\Gamma_{LR}, S_{LR},k}\nonumber\\
    &+&\eta\mu SL_{\Gamma_{LR}, S_{LR},k}+i\mu DL_{\Gamma_{LR}, S_{LR},k}\nonumber\\
    \mathcal{C}_{22}&=&\frac{1}{2}I+K^\top_{S_{LR},k}-i\mu S_{S_{LR},k}\nonumber\\
    \mathcal{D}_2&=&\frac{i}{2\eta}\gamma_{S_{LR}}^{N,1}DL_{C, S_{LR},k}-\frac{1}{2}\gamma_{S_{LR}}^{N,1} S_{C,S_{LR},k}+\frac{\mu}{2\eta}DL_{C,S_{LR},k}+\frac{i\mu}{2}SL_{C,S_{LR},k}.\nonumber
  \end{eqnarray} 
The pair of equations~\eqref{eq:combined1} and~\eqref{eq:combined2} constitutes a linear system of boundary integral equations written in the form
  \begin{equation*}
    \begin{bmatrix}\mathcal{C}_{11}& \mathcal{C}_{12}\\ \mathcal{C}_{21}&\mathcal{C}_{22}\end{bmatrix}\begin{bmatrix}\gamma_{\Gamma_{LR}}^{D,2}w\\\gamma_{S_{LR}}^{N,1}w|_{S_{LR}}\end{bmatrix}=\begin{bmatrix}\mathcal{D}_1\\\mathcal{D}_2\end{bmatrix}\psi.
  \end{equation*}
  We establish the fact that the system of boundary integral equations above is Fredholm of index zero in the space $H^{1/2}(\Gamma_{LR})\times L^2(S_{LR})$. First, we express the operator $\mathcal{C}_{11}$ in the form
  \begin{eqnarray*}
    \mathcal{C}_{11}&=&\mathcal{C}_{11}+\mathcal{C}_{11}\\
    \mathcal{C}_{11}^1&:=&I+K_{\Gamma_{LR},0}-2K_{\Gamma_{LR},0}^2=2\left(\frac{1}{2}I+K_{\Gamma_{LR}, 0}\right)\left(I-K_{\Gamma_{LR}, 0}\right)\\
    \mathcal{C}_{11}^2&:=&2S_{\Gamma_{LR}, \kappa}(N_{\Gamma_{LR}, \kappa}-N_{\Gamma_{LR},k})-i\eta S_{\Gamma_{LR}, k}-i\eta S_{\Gamma_{LR}, \kappa}+2i\eta S_{\Gamma_{LR}, \kappa} K_{\Gamma_{LR}, k}^\top\\
&+&2(K_{\Gamma_{LR},0}-K_{\Gamma_{LR},\kappa})K_{\Gamma_{LR}, \kappa}+2K_{\Gamma_{LR},0}(K_{\Gamma_{LR},0}-K_{\Gamma_{LR}, \kappa})+(K_{\Gamma_{LR}, k}-K_{\Gamma_{LR}, 0}).
  \end{eqnarray*}
  Invoking classical results about the smoothing properties of differences of boundary integral operators~\cite{dominguez2015well} we obtain that the operator $\mathcal{C}_{11}^2:H^{1/2}(\Gamma_{LR})\to H^1(\Gamma_{LR})$ continuously, and thus $\mathcal{C}_{11}^2$ is compact. Also, the operator $\mathcal{C}_{11}^1$ is Fredholm of index zero in the space $H^{1/2}(\Gamma_{LR})$ since (a) the operator $\frac{1}{2}I+K_{\Gamma_{LR}, 0}$ is Fredholm of index zero~\cite{EsFaVer:1992}, (b) the operator $I-K_{\Gamma_{LR}, 0}$ is invertible~\cite{EsFaVer:1992}, and (c) the two operators commute. Thus, the operator $\mathcal{C}_{11}:H^{1/2}(\Gamma_{LR})\to H^{1/2}(\Gamma_{LR})$ is Fredholm of index zero. Similar arguments deliver the fact that the operator $\mathcal{C}_{22}:L^2(S_{LR})\to L^2(S_{LR})$ is also Fredholm of index zero (note that $S_{LR}$ is a union of disjoint Lipschitz domains).  Finally, the kernels of the diagonal operators $\mathcal{C}_{12}:L^2(S_{LR})\to H^{1/2}(\Gamma_{LR})$ and $\mathcal{C}_{21}:H^{1/2}(\Gamma_{LR})\to L^2(S_{LR})$ are smooth as $\Gamma_{LR}\cap S_{LR}=\emptyset$, and thus both those operators are compact. Thus, the matrix operator $\mathcal{C}:=\begin{bmatrix}\mathcal{C}_{11}& \mathcal{C}_{12}\\ \mathcal{C}_{21}&\mathcal{C}_{22}\end{bmatrix}$ is Fredholm of index zero in the space  $H^{1/2}(\Gamma_{LR})\times L^2(S_{LR})$. In order to establish the invertibility of this operators, it therefore suffices to prove its injectivity. The latter, in turn, is settled via duality arguments with respect to the real duality pairings in $L^2(S_{LR})$ and $L^2(\Gamma_{LR})$. The dual of the matrix operator $\mathcal{C}$ can be seen to equal $\mathcal{C}^\top=\begin{bmatrix}\mathcal{C}_{11}^\top & \mathcal{C}_{21}^\top \\ \mathcal{C}_{12}^\top & \mathcal{C}_{22}^\top\end{bmatrix}$ where
\begin{eqnarray}
 \nonumber\\
   \mathcal{C}_{21}^\top&=&-i\eta DL_{\Gamma_{LR},S_{LR},k}+\gamma_{\Gamma_{LR}}^{N,2} DL_{S_{LR},\Gamma_{LR},k}-i\mu(-i\eta SL_{\Gamma_{LR},S_{LR},k}+\gamma_{\Gamma_{LR}}^{N,2}SL_{\Gamma_{LR},S_{LR},k})\nonumber\\
   \mathcal{C}_{12}^\top&=& SL_{S_{LR}, \Gamma_{LR}, k}-2DL_{S_{LR},\Gamma_{LR},k}S_{\Gamma_{LR},\kappa}\nonumber\\
   \mathcal{C}_{22}^\top&=&\frac{1}{2}I+K_{S_{LR},k}-i\mu S_{S_{LR},k}\nonumber.
\end{eqnarray}
Let $(\varphi_0,\psi_0)\in Ker((\mathcal{C})^\top)$ and let us define
\[
v:=SL_{\Gamma_{LR},k}\varphi_0-DL_{\Gamma_{LR},k}[2S_{\Gamma_{LR},\kappa}]\varphi_0-i\mu SL_{S_{LR}, k}\psi_0+DL_{S_{LR},k}\psi_0,\qquad \mathrm{in}\ \mathbb{R}^2\setminus(\Gamma_{LR} \cup S_{LR}).
\]
We have that
\begin{eqnarray}
  \gamma_{\Gamma_{LR}}^{D,2} v &=& S_{\Gamma_{LR}, \kappa}\varphi_0 + S_{\Gamma_{LR},k}\varphi_0 - 2 K_{\Gamma_{LR},k}S_{\Gamma_{LR}, \kappa}\varphi_0\nonumber\\
  &-&i\mu SL_{S_{LR},\Gamma_{LR},k}\psi_0+DL_{S_{LR},\Gamma_{LR},k}\psi_0\nonumber\\
  \gamma_{\Gamma_{LR}}^{N,2} v &=&\frac{1}{2}\varphi_0+K_{\Gamma_{LR},k}^\top \varphi_0-2N_{\Gamma_{LR},k}S_{\Gamma_{LR},\kappa}\varphi_0\nonumber\\
  &-&i\mu \gamma_{\Gamma_{LR}}^{N,2} SL_{S_{LR}, \Gamma_{LR},k}\psi_0+\gamma_{\Gamma_{LR}}^{N,2} DL_{S_{LR}, \Gamma_{LR},k}\psi_0\nonumber.
\end{eqnarray}
The fact that $\mathcal{C}_{11}^\top\varphi_0+\mathcal{C}_{21}^\top\psi_0=0$ translates thus into 
\[
\gamma_{\Gamma_{LR}}^{N,2} v-i\eta \gamma_{\Gamma_{LR}}^{D,2} v=0.
\]
Similarly we have that
\begin{equation}
  \gamma_{S_{LR}}^{D,1} v = \frac{1}{2}\psi_0 +K_{S_{LR}, k}\psi_0-i\mu S_{S_{LR},k}\psi_0 +SL_{\Gamma_{LR},S_{LR}, k}\varphi_0-2DL_{\Gamma_{LR},S_{LR},k}S_{\Gamma_{LR},\kappa}\varphi_0.\nonumber
\end{equation}
The fact that $\mathcal{C}_{12}^\top\varphi_0+\mathcal{C}_{22}^\top\psi_0=0$ translates thus into 
\[
\gamma_{S_{LR}}^{D,1} v =0.
\]
Now $v$ is a solution of Helmholtz equation in $(B_j\cup B_\ell)\setminus (S_j\cup S_\ell)$ satisfying the impedance boundary condition $\gamma_{\Gamma_{LR}}^{N,2} v-i\eta \gamma_{\Gamma_{LR}}^{D,2} v=0$ on $\Gamma_{LR}$ and the Dirichlet boundary condition $\gamma_{S_{LR}}^{D,1} v =0$ on $S_{LR}$. According to the result in Theorem~\ref{tm1} it follows that $v$ is identically 0 in $(B_j\cup B_\ell)\setminus (S_j\cup S_\ell)$ and hence
\begin{equation}\label{bc_S}
\gamma_{S_{LR}}^{D,1} v = 0\qquad \gamma_{S_{LR}}^{N,1}v =0
\end{equation}
  and
\begin{equation}\label{bc_B}
\gamma_{\Gamma_{LR}}^{D,2} v = 0\qquad \gamma_{\Gamma_{LR}}^{N,2}v =0.
\end{equation}  
Using the jump conditions of Dirichlet and Neumann traces across $S_{LR}$ and equation~\eqref{bc_S} we obtain
\[
\gamma_{S_{LR}}^{D,2} v = -\psi_0 \qquad \gamma_{S_{LR}}^{N,2} v = i \mu\psi_0.
\]
Since $v$ is a solution of the Helmholtz equation in the domain $S_j\cup S_\ell$ we have that
\[
\int_{S_{LR}}(|\nabla v|^2 -k |v|^2)dx=-i\mu\int_{S_{LR}} |\psi_0|^2\ ds
\]
which implies that $\psi_0=0$ on $S_{LR}$. Using the jump conditions of Dirichlet and Neumann traces across $\Gamma_{LR}$ and equation~\eqref{bc_B} we obtain
\[
\gamma_{\Gamma_{LR}}^{D,1} v = -2 S_{\Gamma_{LR},\kappa}\varphi_0 \qquad \gamma_{\Gamma_{LR}}^{N,1} v = -\varphi_0.
\]
We have then
\[
\Im{\int_{\Gamma_{LR}}\gamma_{\Gamma_{LR}}^{D,1}v\  \overline{\gamma_{\Gamma_{LR}}^{N,1}v}\ ds}=2\ \Im {\int_{\Gamma_{LR}}S_{\Gamma_{LR},\kappa}\varphi_0\ \overline{\varphi_0}\ ds}
\]
Using the fact that for any closed Lipschitz curve $\Gamma$~\cite{turc2} 
\[
\Im \int_\Gamma (S_{\Gamma,\kappa}\varphi)\ \overline{\varphi}\ ds >0,\quad \varphi\neq 0
\]
when $\Re{\kappa}>0$ and $\Im{\kappa}>0$ we obtain that $v$ is a radiative solution of the Helmholtz equation in $\mathbb{R}^2\setminus (B_j\cup B_\ell)$ satisfying
\[
\Im{\int_{\Gamma_{LR}}\gamma_{\Gamma_{LR}}^{D,1}v\  \overline{\gamma_{\Gamma_{LR}}^{N,1}v}\ ds}\geq 0
\]
and hence $v=0$ in $\mathbb{R}^2\setminus (B_j\cup B_\ell)$. In particular this implies that $\gamma_{\Gamma_{LR}}^{N,1} v =0$, and thus $\varphi_0=0$ on $\Gamma_{LR}$. Consequently, the operator $\mathcal{C}^\top$ is injective, and thus the operator $\mathcal{C}$ is injective as well, which completes the proof of the theorem.
\end{proof}

\section{Representations of the RtR operators $\mathcal{S}^j$ in terms of well conditioned boundary integral operators\label{cfie}}

Our goal is to derive an explicit expression of subdomain RtR operators in terms of well-conditioned boundary integral operators. We mention that alternative robust boundary integral formulations for solutions of Helmholtz equations with Robin boundary conditions that feature in DDM were introduced in~\cite{steinbach2011stable}.

\subsection{Calculation of the RtR maps via boundary integral operators\label{RtR}}

For the sake of ease of exposition we focus on the simplified setting from Section~\ref{well-posed}, that is one scatterer $S$ surrounded by a box $B$; extensions to multiple scatterers inside $B$ is straightforward. Applying Green identities in the domain $ B\setminus S$ we get
\[
u(\mathbf{x}) = SL_{\partial B,k}\gamma_{\partial B}^N u - DL_{\partial B,k}\gamma_{\partial B}^D u-SL_{\partial S,k}\gamma_{\partial S}^N u,\qquad \mathbf{x}\in B\setminus S
\]
where $\gamma_{\partial S}^N$ denotes the Neumann trace with respect to the normal on $\partial S$ exterior to $ S$ applied to functions defined in the domain $ B\setminus S$. We replace $\gamma_{\partial B}^Nu=i\eta\gamma_{\partial B}^Du+f$ in the equation above and obtain
\begin{equation}\label{Green1}
  u(\mathbf{x}) = [i\eta\ SL_{\partial B,k}- DL_{\partial B,k}]\gamma_{\partial B}^D u-SL_{\partial S,k}\gamma_{\partial S}^N u + SL_{\partial B,k} f,\qquad \mathbf{x}\in B\setminus S.
  \end{equation}
The main idea is to apply Dirichlet and Neumann traces of equation~\eqref{Green1} on the boundaries $\partial B$ and $\partial S$ respectively, and to combine these traces in a regularized combined manner on $\partial B$ and in the classical combined manner of Burton-Miller on $\partial S$. We first apply the Dirichlet trace on $\partial B$ on both sides of equation~\eqref{Green1} and obtain
\begin{equation}\label{Dirichlet_ext}
  \frac{1}{2}\gamma_{\partial B}^D u-[i\eta S_{\partial B,k}-K_{\partial B,k}]\gamma_{\partial B}^D+SL_{\partial S,\partial B,k}\gamma_{\partial S}^Nu=S_{\partial B,k}f, 
  \end{equation}
where $SL_{\partial S,\partial B,k}\psi$ denotes the single layer potential applied to the density $\psi$ defined on $\partial S$ and evaluated on $\partial B$. Similarly, we apply the Neumann trace on $\partial B$ on both sides of equation~\eqref{Green1} and obtain
\begin{equation}\label{Neumann_ext}
  \frac{i\eta}{2}\gamma_{\partial B}^D u-[i\eta K^\top_{\partial B,k}-N_{\partial B,k}]\gamma_{\partial B}^D+\gamma_{\partial B}^NSL_{\partial S,\partial B,k}\gamma_{\partial S}^Nu=-\frac{1}{2}f+K^\top_{\partial B,k}f.
\end{equation}
We combine equation~\eqref{Dirichlet_ext} with equation~\eqref{Neumann_ext} preconditioned on the left by $-2S_{\partial B,k+i\varepsilon}$ with $\varepsilon>0$ and obtain
\begin{eqnarray}\label{reg_combined_ext}
  \mathcal{A}_{\partial B,\partial B}\gamma_{\partial B}^D u+\mathcal{A}_{\partial B,\partial S}\gamma_{\partial S}^N u&=&\left(S_{\partial B,k+i\varepsilon}-2S_{\partial B,k+i\varepsilon}K^\top_{\partial B,k}+S_{\partial B,k}\right)f\nonumber\\
  \mathcal{A}_{\partial B,\partial B}&:=&\frac{1}{2}I-2S_{\partial B,k+i\varepsilon}N_{\partial B,k}-i\eta S_{\partial B,k+i\varepsilon}-i\eta S_{\partial B,k}+K_{\partial B, k}\nonumber\\
  &+&2i\eta S_{\partial B,k+i\varepsilon}K^\top_{\partial B,k}\nonumber\\
\mathcal{A}_{\partial B,\partial S}&:=& SL_{\partial S,\partial B,k} -2S_{\partial B,k+i\varepsilon}\gamma_{\partial B}^NSL_{\partial S,\partial B,k}.
\end{eqnarray}
We now turn to applying traces of equation~\eqref{Green1} on $\partial S$ and combining them in the usual Burton-Miller manner. First we apply the Dirichlet trace on $\partial S$ to equation~\eqref{Green1} and obtain
\begin{equation}\label{Dirichlet_int}
  (-i\eta SL_{\partial B,\partial S,k}+DL_{\partial B,\partial S,k})\gamma_{\partial B}^Du+S_{\partial S,k}\gamma_{\partial S}^Nu=SL_{\partial B,\partial S,k}f
\end{equation}
where $SL_{\partial B,\partial S,k}\varphi$ denotes the single layer potential applied to the density $\varphi$ defined on $\partial B$ and evaluated on $\partial S$; the meaning of the notation for the double layer potential $DL_{\partial B,\partial S,k}$ is similar.  Applying the Neumann trace on $\partial S$ to equation~\eqref{Green1} we obtain
\begin{equation}\label{Neumann_int}
  (-i\eta \gamma_{\partial S}^N SL_{\partial B,\partial S,k}+\gamma_{\partial S}^N DL_{\partial B,\partial S,k})\gamma_{\partial B}^Du+\left(\frac{1}{2}I+K^\top_{\partial S,k}\right)\gamma_{\partial S}^N u=\gamma_{\partial S}^N SL_{\partial B,\partial S,k}f.
\end{equation}
We combine equation~\eqref{Neumann_int} and equation~\eqref{Dirichlet_int} multiplied by $-i\mu,\ \mu\neq 0$ and obtain
\begin{eqnarray}\label{CFIE_omega_0}
  \mathcal{A}_{\partial S,\partial B}\gamma_{\partial B}^D u+\mathcal{A}_{\partial S,\partial S}\gamma_{\partial S}^Nu&=&(\gamma_{\partial S}^N SL_{\partial B,\partial S,k}-i\mu SL_{\partial B,\partial S,k})f\nonumber\\
  \mathcal{A}_{\partial S,\partial B}&:=&-i\eta \gamma_{\partial S}^N SL_{\partial B,\partial S,k}+\gamma_{\partial S}^N DL_{\partial B,\partial S,k}\nonumber\\&-&i\mu(-i\eta SL_{\partial B,\partial S,k}+DL_{\partial B,\partial S,k})\nonumber\\
 \mathcal{A}_{\partial S,\partial S}&:=&\frac{1}{2}I+K^\top_{\partial S,k}-i\mu S_{\partial S,k}. 
\end{eqnarray}
\begin{remark}
  The cases of other types of boundary conditions on $\partial S$ can be treated by suitably combining Dirichlet and Neumann traces, possibly  using regularizing operators according to the prescriptions in~\cite{turc1,turc_corner_N}, so that to formulate a direct well conditioned boundary integral equation on $\partial S$. Different boundary conditions call for different types of unknowns on $\partial S$ (e.g. Neumann boundary conditions call for $\gamma_{\partial S}^D u$ as an unknown, etc.). In the case when $S$ is an open curve, only one type of traces is used, and the resulting boundary integral equations are preconditioned according to the methodology presented in~\cite{lintner}.
\end{remark}

Combining then equations~\eqref{reg_combined_ext} and~\eqref{CFIE_omega_0} we get the following system of boundary integral equations for the unknowns $\gamma_{\partial B}^Du$ and $\gamma_{\partial S}^Nu$:
\begin{equation}\label{system_S}
  \begin{bmatrix}\mathcal{A}_{\partial B,\partial B} & \mathcal{A}_{\partial B,\partial S} \\ \mathcal{A}_{\partial S,\partial B} & \mathcal{A}_{\partial S,\partial S}\end{bmatrix}\begin{bmatrix}\gamma_{\partial B}^Du \\ \gamma_{\partial S}^Nu\end{bmatrix}=\begin{bmatrix} S_{\partial B,k+i\varepsilon}-2S_{\partial B,k+i\varepsilon}K^\top_{\partial B,k}+S_{\partial B,k}\\ \gamma_{\partial S}^N SL_{\partial B,\partial S,k}-i\mu SL_{\partial B,\partial S,k}\end{bmatrix}f.
  \end{equation}
We state a central result whose proof follows along the same arguments as in the proof of Theorem~\ref{thm_inv}:
\begin{theorem}\label{thm2}
The system of equations~\eqref{system_S} has a unique solution in the space $L^2(\partial B)\times L^2(\partial S)$ for any data $f\in L^2(\partial B)$. The solution of this system of boundary integral equations depends continuously on the data $f$.
\end{theorem}
If we denote
\[
 \begin{bmatrix}\mathcal{A}_{\partial B,\partial B} & \mathcal{A}_{\partial B,\partial S} \\ \mathcal{A}_{\partial S,\partial B} & \mathcal{A}_{\partial S,\partial S}\end{bmatrix}^{-1}= \begin{bmatrix}\mathcal{S}_{\partial B,\partial B} & \mathcal{S}_{\partial B,\partial S} \\ \mathcal{S}_{\partial S,\partial B} & \mathcal{S}_{\partial S,\partial S}\end{bmatrix}
 \]
 it follows that an explicit representation of the RtR operator $\mathcal{S}$ is given in the form 
 \begin{eqnarray}\label{S_explicit}
   \mathcal{S}&=& I+2i\eta\mathcal{S}_{\partial B,\partial B}(S_{\partial B,k+i\varepsilon}-2S_{\partial B,k+i\varepsilon}K^\top_{\partial B,k}+S_{\partial B,k})\nonumber\\
     &+& 2i\eta \mathcal{S}_{\partial B,\partial S}(\gamma_{\partial S}^N SL_{\partial B,\partial S,k}-i\mu SL_{\partial B,\partial S,k}).
 \end{eqnarray}
 In addition, if we denote by $Y$ the operator that maps the Robin data $\gamma_{\partial B}^Nu-i\eta\gamma_{\partial B}^D u$ to $\gamma_{\partial S}^N u$, this operator can also be computed in explicit form
\begin{eqnarray}\label{Y_explicit}
  Y&=&\mathcal{S}_{\partial S,\partial B}(S_{\partial B,k+i\varepsilon}-2S_{\partial B,k+i\varepsilon}K^\top_{\partial B,k}+S_{\partial B,k})\nonumber\\
     &+& \mathcal{S}_{\partial S,\partial S}(\gamma_{\partial S}^N SL_{\partial B,\partial S,k}-i\mu SL_{\partial B,\partial S,k}).
 \end{eqnarray}

\subsection{Calculation of the exterior RtR operator $\mathcal{S}^{ext}$}

We apply Green's identities in $\mathbb{R}^2\setminus B$ for the scattered field
\[
u^s = -SL_{\partial B,k}\gamma_{\partial B}^{N,1}u^s+DL_{\partial B,k}\gamma_{\partial B}^{D,1}u^s
\]
where $\gamma_{\partial B}^{D,1}$ and $\gamma_{\partial B}^{N,1}$ denote the Dirichlet and respectively Neumann traces in the domain exterior to $B$. We replace $\gamma_{\partial B}^{N,1}u^s=i\eta\gamma_{\partial B}^{D,1}u^s+\varphi$ in the equation above and obtain
\begin{equation}\label{eq:repr_Ext}
  u^s(\mathbf{x})=[-i\eta SL_{\partial B,k}+DL_{\partial B,k}]\gamma_{\partial B}^{D,1}u^s-SL_{\partial B,k}\varphi,\qquad \mathbf{x}\in\mathbb{R}^2\setminus B.
  \end{equation}
Applying the exterior Dirichlet and Neumann traces on $\partial B$ to equation~\eqref{eq:repr_Ext} we obtain
\begin{eqnarray}\label{eq:traces1}
\frac{1}{2}\gamma_{\partial B}^{D,1} u^s - K_{\partial B,k}(\gamma_{\partial B}^{D,1} u^s) +i\eta S_{\partial B,k}(\gamma_{\partial B}^{D,1}u^s)&=&-S_{\partial B,k}\varphi\nonumber\\
\frac{i\eta}{2}\gamma_{\partial B}^{D,1} u^s+i\eta K_{\partial B,k}^\top\gamma_{\partial B}^{D,1} u^s-N_{\partial B,k}\gamma_{\partial B}^{D,1} u^s&=&-\frac{1}{2}\varphi-K_{\partial B,k}^\top\varphi.\nonumber
\end{eqnarray}
Following the strategy introduced in~\cite{turc_corner_N} we add the first equation above to the second equation above composed on the left with the operator $-2S_{\partial B,k+i\varepsilon},\ \varepsilon>0$ and we obtain a representation of the operator $\mathcal{S}^{ext}$ that involves well conditioned boundary integral operators
\begin{eqnarray}\label{eq:CFIER1}
\mathcal{S}^{ext}&=&I+2i\eta (\mathcal{A}_{\partial B,\partial B}^{ext})^{-1}\left(S_{\partial B,k+i\varepsilon}+2S_{\partial B,k+i\varepsilon}K^\top_{\partial B,k}-S_{\partial B,k}\right)\\
\mathcal{A}_{\partial B,\partial B}^{ext}&:=&\frac{1}{2}I-2S_{\partial B,k+i\varepsilon} N_{\partial B,k}-i\eta S_{\partial B,k+i\varepsilon} +i\eta S_{\partial B,k} - K_{\partial B, K}- 2i\eta S_{\partial B,k+i\varepsilon} K_{\partial B, k}^\top. \nonumber
\end{eqnarray}

\subsection{High-order Nystr\"om discretizations of RtR maps~\label{Nystrom}}

\parskip 10pt plus2pt minus1pt
\parindent0pt 

We use a Nystr\"om discretization of the RtR maps computed as in equations~\eqref{S_explicit} and~\eqref{eq:CFIER1} that relies on discretizations of the four boundary integral operators in the Calder\'on calculus. The latter, in turn, rely on (a) use of graded meshes based on sigmoid transforms that cluster polynomially discretization points toward corners, (b) splitting of the kernels of the parametrized versions of the boundary integral operators that feature in equations~\eqref{S_explicit} and~\eqref{eq:CFIER1} into sums of regular quantities and products of periodized logarithms and regular quantities, (c) trigonometric interpolation of the densities of the boundary integral operators, and (d) analytical expressions for the integrals of products of periodic singular and weakly singular kernels and Fourier harmonics. These discretizations were introduced in~\cite{turc2016well} where the full details of this methodology were presented. The main idea of our Nystr\"om discretization is to incorporate sigmoid transforms~\cite{KressCorner} in the parametrization of a closed Lipschitz curve $\Gamma$ and then split the kernels of the Helmholtz boundary integral operators into smooth and singular components. Using graded meshes that avoid corner points and classical singular quadratures of Kusmaul and Martensen~\cite{kusmaul,martensen}, we employ the Nystr\"om discretization presented in~\cite{turc2016well} to produce high-order approximations of the boundary integral operators that enter equations~\eqref{S_explicit} and~\eqref{eq:CFIER1}.

The Nystr\"om discretization of the Helmholtz boundary integral operators presented above delivers naturally a discretization of the RtR operators $\mathcal{S}^j$ per equations~\eqref{S_explicit}. Indeed, for each of the subdomains $B_j,\ j=1,\ldots,L$ we employ graded meshes $\mathbf{x}^j_{m},m=0,\ldots,N_j-1$ on $\partial B_j$ together with appropriate meshes on the scatterers inside the subdomain $B_j$ (say of size $M_j$) which we use to discretize all the boundary integral operators that feature in equations~\eqref{S_explicit} according to the prescriptions above.  The discretization of the RtR map corresponding to $\partial B_j$ is constructed then as a $N_j\times N_j$ collocation matrix  $\mathcal{S}^j_N$. We note that formula~\eqref{S_explicit} also features inverses of boundary integral operators, whose discretization is obtained through direct linear algebra solvers. We note that the cost of obtaining the collocation matrix $\mathcal{S}_N^j$ is $\mathcal{O}((N_j+M_j)^3)$. Thus, the subdomain decomposition of the computational domain $B_0$ has to be performed with care so that the size of subdomain discretizations $N_j+M_j$ is amenable to direct linear algebra solvers. Once the collocation matrices $\mathcal{S}^j_N$ are constructed, the discretization of the interface subdomain RtR maps (e.g. the maps $\mathcal{S}^j_{j\ell,\ell j}$ and all the other ones defined in equation~\eqref{split1}) is straightforward since it simply amounts to extracting suitable blocks from the matrices $\mathcal{S}^j_N$. Indeed, the discretization of the operators $\mathcal{S}^j_{j\ell,\ell j}$ consists of extracting from the collocation matrix $\mathcal{S}^j_N$ the block that corresponds to self-interactions of the grid points on the interface/edge $\partial B_j\cap\partial B_\ell$ (this is possible since none of these mesh points $\mathbf{x}^j_m$ corresponds to a corner of $B_j$). We also mention that all the matrix inversions needed in the Schur complement algorithm (cf. formula~\eqref{inv_matrix_explicit}) are performed by direct linear algebra methods as well.

The elimination of the interface unknowns $g_{j\ell}$ from the linear system~\eqref{system_A} via Schur complements is performed in a hierarchical fashion that optimizes the computational cost of the linear algebra manipulations. In a nutshell, and referring to the case depicted in Figure~\ref{fig:4domains}, the collocated values of the Robin data unknowns $(g_{j\ell},g_{\ell,j})$ and $(g_{j'\ell'},g_{\ell'j'})$ (and all their counterparts) are eliminated in the first stage, then the collocated values of the lumped Robin data unknowns $(g_{jj'\ell\ell'},g_{j'j\ell'\ell})$ (together with their counterparts) are eliminated in the second stage, and the procedure is repeated until all the Robin interface unknowns are eliminated. Equivalently, the RtR maps corresponding to the subdomains $B_j$ are merged hierarchically: in the first stage the discretizations of the RtR maps corresponding to the subdomains $B_j$ and $B_\ell$ as well as $B_{j'}$ and $B_{\ell'}$ (and all of their counterparts) are merged to produce discretizations of the RtR maps corresponding to the subdomains $B_{j\cup\ell}$, $B_{j'\cup\ell'}$, etc.; in the second stage the merging procedure is in turn applied to the discretizations of the RtR maps for the subdomains $B_{j\cup\ell}$ and $B_{j'\cup\ell'}$ (and all of their counterparts) to deliver discretizations of the RtR maps corresponding to the subdomains $B_{j\cup\ell\cup j'\cup\ell'}:=B_j\cup B_\ell\cup B_{j'}\cup B_{\ell'}$ (and similar other subdomains); the merging procedure is repeated until a discretization of the RtR map corresponding to the interior domain $B_0$ is calculated.

We present in the next Section numerical results obtained from our Schur complement DD algorithm. 

\section{Numerical results}

In this section we present a variety of numerical results that highlight the performance of our Schur complement DDM algorithm for solution of multiple scattering problems. All the results presented here were produced on a single core (3.7 GHz Intel Xeon processor) of a MacPro machine with 64Gb of memory by a MATLAB implementation of our algorithm. We present results for scattering from clouds of sound-soft scatterers (e.g. Dirichlet boundary conditions on the scatterers). The extension to other types of boundary conditions is straightforward. We create clouds of scatterers by choosing a large box that we subdivide into $L$ subdomains (boxes) and then we place inside each subdomain $P$ scatterers whose position is random, while ensuring that the scatterers do not intersect each other and do not intersect the boundary of the domain. In all the experiments in this section we used $\eta=k$ in the RtR maps.

Our DD algorithm proceeds in two stages: an offline (precomputation) stage whereby all the subdomain RtR maps are computed using the Nystr\"om discretization presented in Section~\ref{Nystrom}, and a stage where the linear system~\eqref{system_A} is solved via hierarchical Schur complements. Finally, in the solution stage, we solve a linear system involving a dense matrix that corresponds to connecting the unknowns on the inner/outer artificial boundary through interior and respectively exterior RtR maps. We note that although the algorithm is highly parallelizable, our current implementation does not take advantage of these possibilities.

We comment next on the computational complexity of our Schur complement elimination algorithm. Assuming a collection of $L=\ell_1\times \ell_2$ of identical square subdomains, each one containing $P_L$ scatterers inside. If $n_P$ collocation points are used per each scatterer to resolve the solution (say that these amount to 6 pts/wavelength which is typical for Nystr\"om discretizations of boundary integral equations) then we argue that about $P_L^{1/2} n_P^{1/2}$ collocation points are needed per side of each subdomain. Since there are overall $2\ell_1\ell_2+\ell_1+\ell_2$ common interfaces in the DD algorithm, the discretization of the linear system~\eqref{system_A} would require about $2(2\ell_1\ell_2+\ell_1+\ell_2)P_L^{1/2} n_P^{1/2}$ unknowns (recall that there are two unknowns per interface). However, the matrix corresponding to the linear system~\eqref{system_A}, although sparse, is never stored in practice, and the solution of that system is performed by employing hierarchical Schur complements of small size. The cost of our Schur complement elimination algorithm is thus dominated by that of the solution of the linear system~\eqref{system_A_reduced} that features a dense matrix corresponding to unknowns on the inner/outer interface, and as such the cost is $\mathcal{O}((2(\ell_1+\ell_2)P_L^{1/2} n_P^{1/2})^3)$. Thus, if we denote $N_T=\ell_1\ell_2 P_Ln_P$, the computational cost of our Schur complement solver is roughly $\mathcal{O}((4N_T)^{3/2})$. In addition, the precomputation/offline stage of our algorithm requires a computational cost of $\mathcal{O}(\ell_1\ell_2(P_Ln_P)^3)$ in order to compute the $L$ subdomain RtR maps, assuming that the distribution of scatterers inside each box is different. Nevertheless, in the important case of photonic crystal applications, the distribution of scatterers inside the subdomains is identical, in which instance the precomputation cost can be significantly reduced. The cost of computing a single subdomain RtR map can be further reduced if fast compression algorithms such as $\mathcal{H}$-matrices are used. In contrast, the cost of a direct boundary integral solver for the solution of the multiple scattering problem would be $\mathcal{O}(N_T^3)$ with a $\mathcal{O}(N_T^2)$ amount of memory needed. Consequently, in multiple scattering applications that involve very large numbers of scatterers, the direct approach is simply too costly. In case Krylov subspace iterative solvers are employed for the solution of the very large linear algebra problem resulting from the direct approach, the numbers of iterations is prohibitive. Clearly, our algorithm is competitive when the number of scatterers per subdomain (i.e. $P_L$) is large. We emphasize that our DD algorithm is a direct method, and as such multiple incidence can be treated with virtually no additional cost.

We present in Table~\ref{comp1} a comparison between the global BIE approach and our DD algorithm. The multiple scattering configuration in this experiment consists of a cloud of 640 lines segment scatterers, a configuration that is challenging to volumetric discretizations (e.g. finite differences, finite elements). More precisely, our configuration is enclosed by a square box of size 16 by 16 which is divided in a collection of $4\times 4$ subdomains, each a a square box of size 4 in which we placed  a collection of 40 line segments of length 0.4 whose centers and orientations are chosen randomly (yet avoiding self intersections and intersections with the boundary of the box). The distribution of scatterers is different in each subdomain, and thus the subdomain RtR maps are different. In all the numerical results presented, we report in the column ``Unknowns'' the total number of unknowns needed to discretize the scatterers in the cloud; in the column ``Unknowns DD'' we report the number of unknowns in the DD linear system~\eqref{system_A} and the number of unknowns in the reduced system~\eqref{system_A_reduced}. We emphasize that the matrix related to the DD system~\eqref{system_A} is not stored, it is only the matrix in the reduced system~\eqref{system_A_reduced} obtained after applications of the Schur complements that is stored. Our DD algorithm uses $4\times 4$ subdomains. We chose a wavenumber $k=8$ such that the scattering ensemble has size $20\lambda \times 20\lambda$ and we compared the far-field results produced by each method, and we observe excellent agreement. As it can be seen from the results in Table~\ref{comp1}, our solver is more competitive than the solver based on the global first kind boundary integral equation (BIE) formulation of the multiple scattering problem, even when accounting for the offline cost. We present in Figure~\ref{fig:cloud} a depiction of the total field in a neighborhood of the scatterer cloud.

\begin{table}
\begin{center}
\resizebox{!}{0.8cm}
{
\begin{tabular}{|c|c|c|c|c|c|c|c|c|c|c|c|c|}
\hline
Unknowns & BIE solver & Unknowns DD & Offline & \multicolumn{3}{c|}{DDM} & Error far-field\\
\cline{5-7}
 & & & & Hierarchical elimination & Solution & Total time\\
\hline
5,120 & 13.24 & 2,560/512 & 2.00 & 0.12 & 0.22 & 0.34 & 1.0 $\times$ $10^{-1}$\\
10,240 & 66.56 & 3,840/768 & 5.90 & 0.18 & 0.47 & 1.05 & 3.1 $\times$ $10^{-3}$\\
20,480 & 419.74 & 5,760/1,152 & 21.61 & 0.42 & 1.10 & 1.52 & 9.9 $\times$ $10^{-5}$\\
\hline
\end{tabular}
}
\caption{Comparison of our Schur complement DD solver with the BIE direct solver for a $20\lambda\times20\lambda$ cloud of 640 line segment scatterers, each scatterer being of size $0.4\lambda$. The DD uses 16 subdomains with 40 scatterers in each subdomain.  The offline time refers to the time needed to compute the four subdomain RtR maps, which were computed sequentially. Solution time is the time it takes to solve the final linear system after the unknowns on the interior interfaces have been eliminated. All times are in seconds. The BIE iterative solver required 912 iterations to reach a GMRES tolerance of $10^{-6}$ in the case of the coarsest discretization. Given that the BIE formulation is of the first kind, the numbers of GMRES iterations grow with the size of the discretization. \label{comp1}}
\end{center}
\end{table}

\begin{figure}
\centering
\includegraphics[width=80mm]{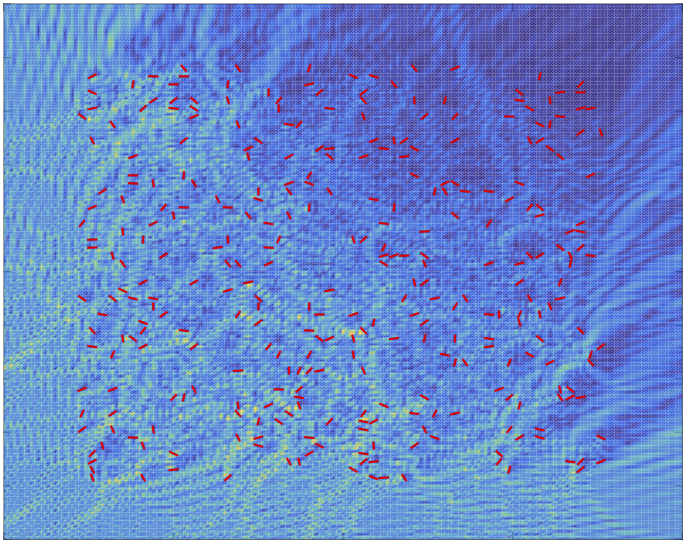}
\caption{Total field scattered by a cloud of 640 line segment scatterers as an incident plane wave making an angle of $45^\circ$ with the vertical impinges on the cloud.}
\label{fig:cloud}
\end{figure}

We present in Table~\ref{comp2} an illustration of the performance of our algorithm for large clouds of scatterers (e.g. made up of 10,240 and respectively 40,960 scatterers) that span domains of size $80\lambda\times 80\lambda$ and respectively $160\lambda \times 160\lambda$, each scatterer being of size $0.4\lambda$. Again, the arrangement of scatterers in the subdomains was produced in the same manner as in Table~\ref{comp1} (there are $16\times 16$ and respectively $32\times 32$ subdomains), and the distribution of scatterers is different in each subdomain. These configurations could model rain drops or possibly foliage. Given the large size of the cloud, global BIE based methods are beyond the limits of the computational resources we used in these experiments.  The number of collocation points used for the discretization of the RtR maps was chosen to be $\mathcal{O}(N_T^{1/2})$ where $N_T$ is the number of discretization points needed on the scatterers.  We present in Figure~\ref{fig:RCS} Radar Cross Section (RCS) plots (in dB) for (a) the configuration used in Table~\ref{comp2} and (b) for the same geometric arrangement but doubling the frequency (this makes the cloud of scatterers to span a domain of size $160\lambda\times 1600\lambda$) when a plane wave whose direction is making a $45^\circ$ angle with the vertical impinges on the ensemble of scatterers.

\begin{table}
\begin{center}
\resizebox{!}{1.0cm}
{
\begin{tabular}{|c|c|c|c|c|c|c|c|c|c|c|c|c|}
\hline
Size & Unknowns& Unknowns DDM & Offline & \multicolumn{3}{c|}{DDM} & Error far-field\\
\cline{5-8}
 & & & & Hierarchical elimination & Solution & Total time\\
\hline
10,240/80$\lambda$ $\times$ 80$\lambda$ & 81,920 & 34,816/2,048 & 26.6 & 3.2 & 4.1 & 7.3 & 7.3 $\times$ $10^{-1}$\\
10,240/80$\lambda$ $\times$ 80$\lambda$ & 163,840 & 52,224/3,072 & 88.9 & 8.9 & 10.6 & 19.5 & 6.5 $\times$ $10^{-2}$\\
10,240/80$\lambda$ $\times$ 80$\lambda$ & 327,680 & 78,336/4,608 & 337.1 & 25.4 & 30.4 & 55.8 & 6.9 $\times$ $10^{-3}$\\
10,240/80$\lambda$ $\times$ 80$\lambda$ & 655,360 & 117,504/6,912 & 1,388 & 79.7 & 85.1 & 164.8 & 4.2 $\times$ $10^{-6}$\\
\hline
40,960/160$\lambda$ $\times$ 160$\lambda$ & 1,310,720 & 304,128/6,144 & 1,473 & 208.1 & 197.8 & 405.9 & 6.4 $\times$ $10^{-3}$\\
\hline
\end{tabular}
}
\caption{Performance of our Schur complement DD solver for the multiple scattering off of (1) a cloud of 10,240 line segment scatterers each of size $0.4\lambda$ spanning $80\lambda\times 80\lambda$; and (2) a cloud of 40,960 line segment scatterers each of size $0.4\lambda$ spanning $160\lambda\times 1600\lambda$ . The DD algorithm uses $16\times 16$ and respectively $32\times 32$ subdomains with 40 scatterers in each subdomain. All times are in seconds. The reference solutions were obtained using $655,360$ unknowns to discretize the scatterers and $139,264$ DD unknowns in the case (1) and $2,621,440$ unknowns to discretize the scatterers and $456,192$ DD unknowns in case (2). The computational times to produce the reference solution in case (2) was 5,372 sec for the offline computations, and 1,119 sec for the DD algorithm.\label{comp2}}
\end{center}
\end{table}
%
%

\begin{figure}
\centering
\includegraphics[width=55mm]{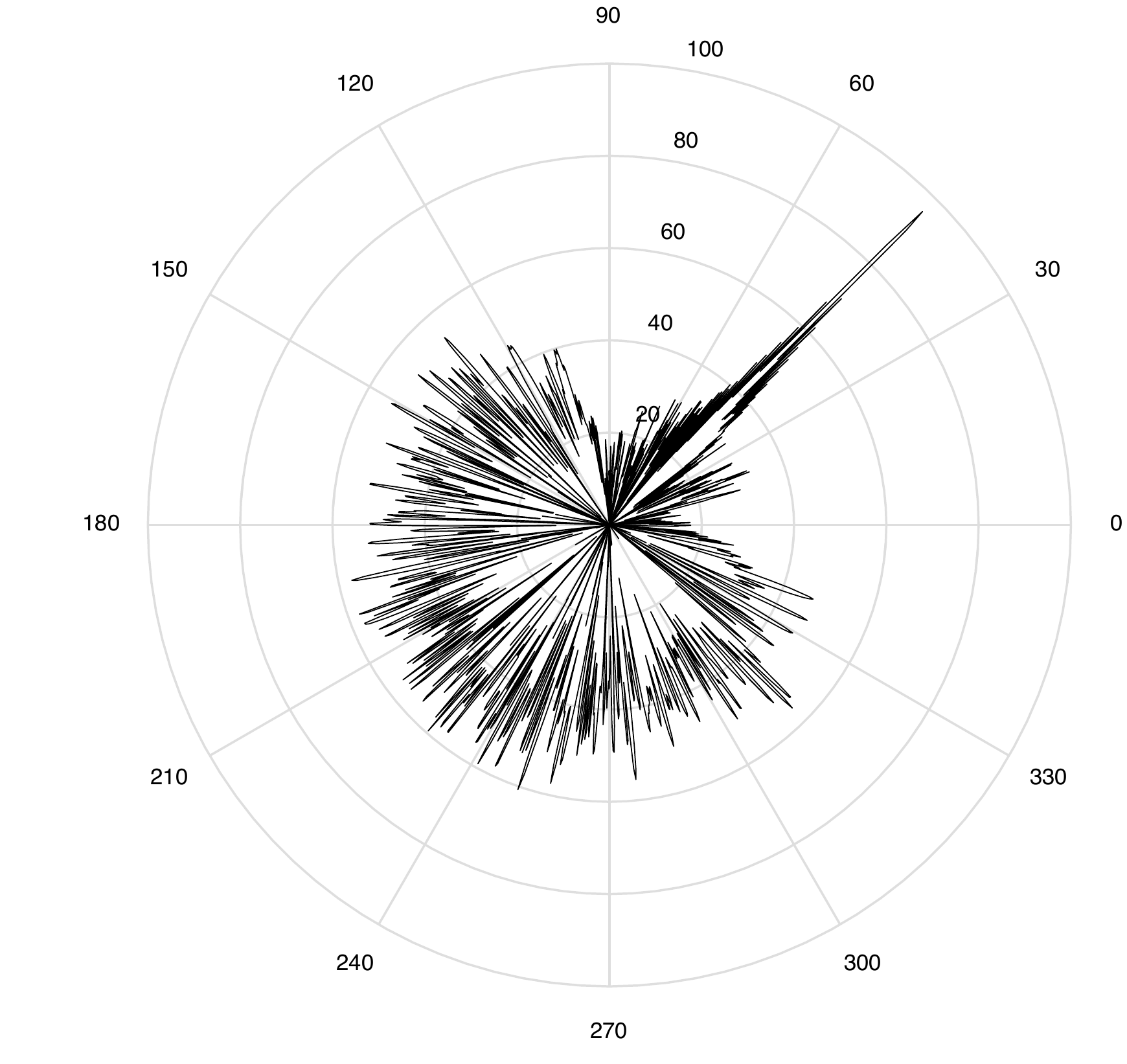}\includegraphics[width=55mm]{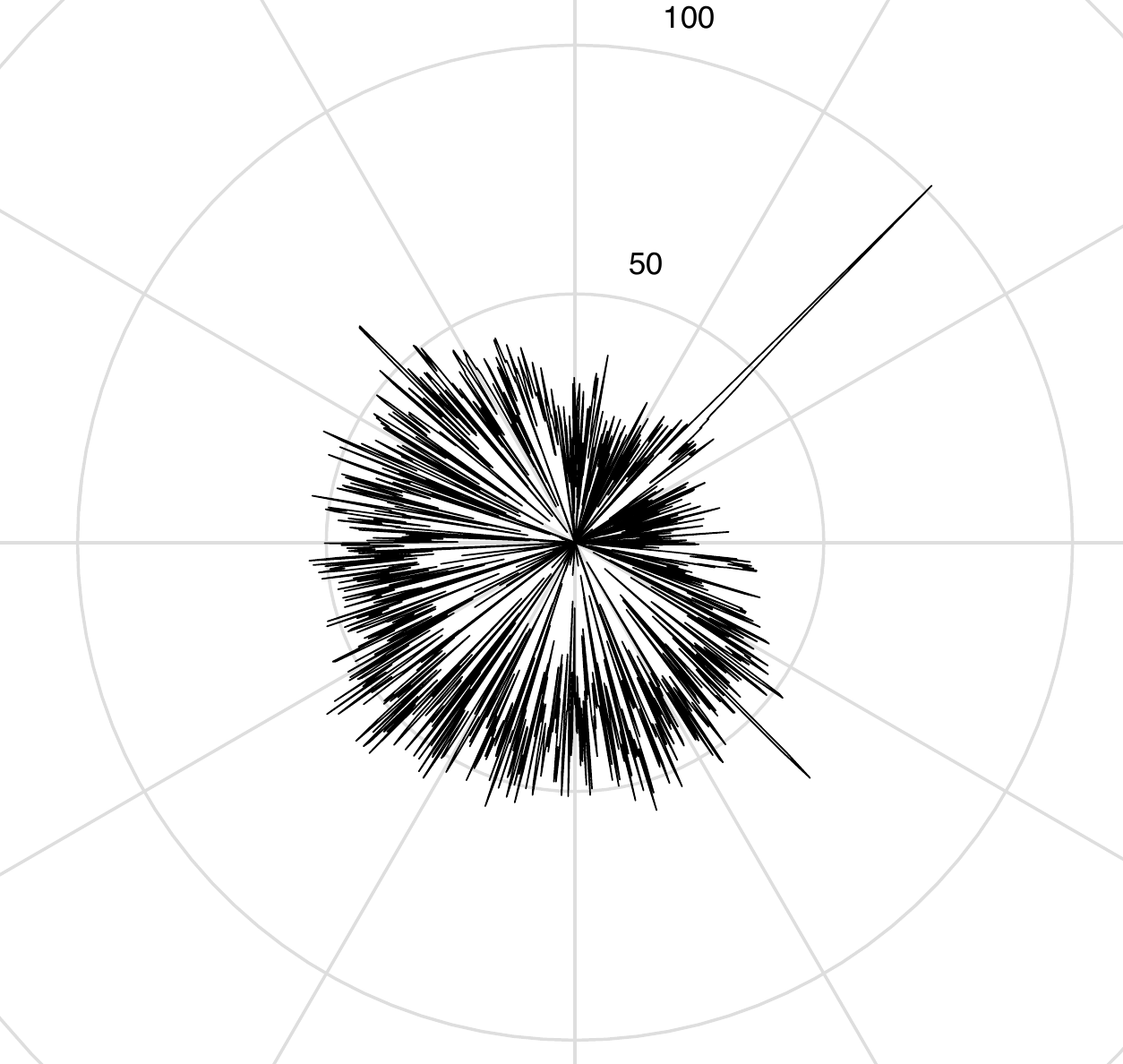}\includegraphics[width=55mm]{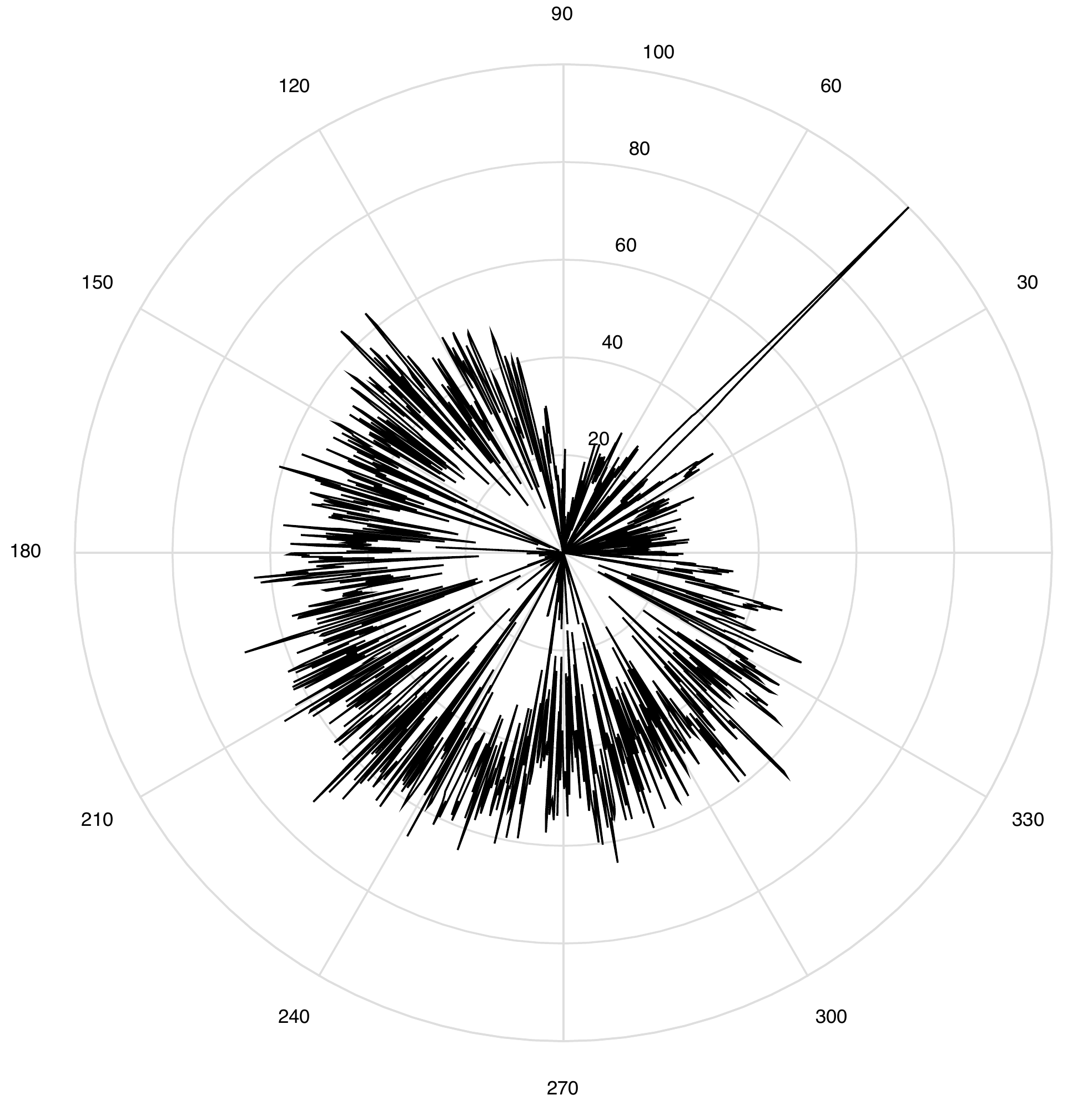}
\caption{RCS for 10,240 line segment scatterers occupying regions of sizes (a) $80\lambda\times 80\lambda$ (left), each scatterer is about 0.4$\lambda$; (b) $160\lambda\times 160\lambda$ (center), each scatterer is about 0.8$\lambda$; and (c)  40,960 line segment scatterers occupying a region of size $160\lambda\times 160\lambda$ (right), each scatterer is about 0.4$\lambda$; for a plane wave incident field making an angle of $45^\circ$ with the $y$-axis.}
\label{fig:RCS}
\end{figure}

We conclude with an illustration in Figure~\ref{fig:Photonic} of the performance of our DD solver for simulation of wave propagation in photonic crystal like structures such as those depicted therein. The geometric configuration is made up of a collection of circles such that the distance between them equals their diameter and a channel defect $0.6\lambda$.
\begin{figure}
  \centering
  \includegraphics[width=55mm]{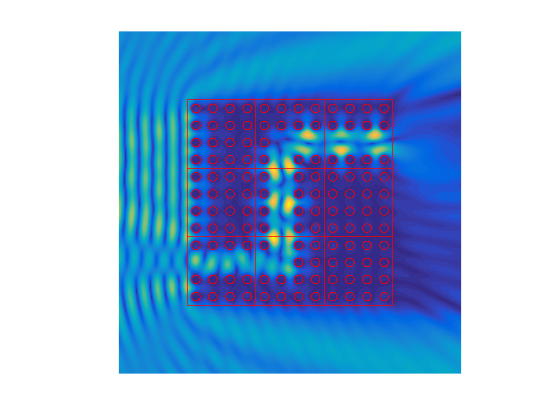}\includegraphics[width=55mm]{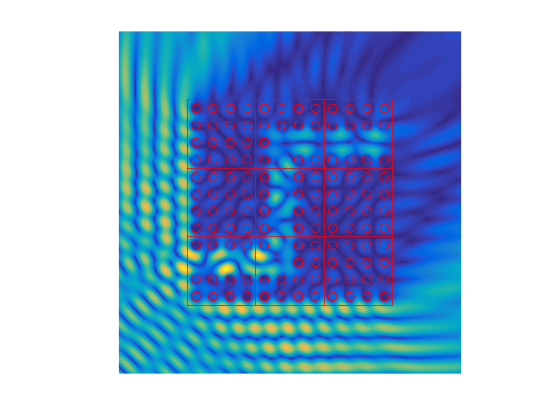}\includegraphics[width=55mm]{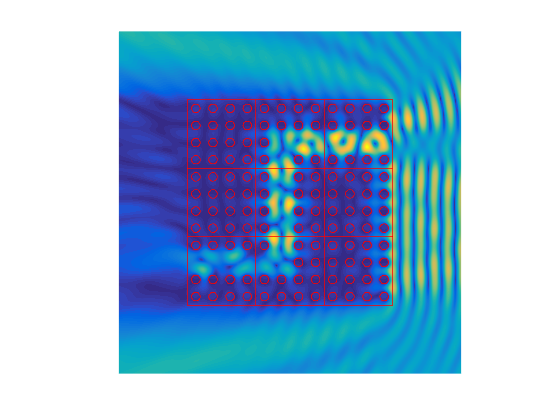}\\
\caption{Simulation of propagation through a channel defect and partitioning in subdomains. Computational cost of our DD algorithm was 3 sec. The angle of incidence was $-90^\circ$ (left), $45^\circ$ (center), and $90^\circ$ (right) with the vertical $y$-axis.}
\label{fig:Photonic}
\end{figure}

\section{Conclusions}

We presented a Schur complement DD solver based on integral equations for the solution of two dimensional frequency domain multiple scattering problems. Our algorithm provides a direct solver for the solution of large multiple scattering problems for which the direct BIE approach is out of reach. Extensions to three dimension configurations are currently underway. 

\section*{Acknowledgments}
Catalin Turc gratefully acknowledge support from NSF through contract DMS-1312169. Yassine Boubendir gratefully acknowledge support from NSFthrough contract DMS-1319720.

\bibliography{biblioM}

\begin{thebibliography}{10}

\bibitem{adams:2003}
R.A. Adams and J.J.F. Fournier.
\newblock {\em Sobolev spaces}, volume 140 of {\em Pure and Applied Mathematics
  (Amsterdam)}.
\newblock Elsevier/Academic Press, Amsterdam, second edition, 2003.

\bibitem{turc_corner_N}
A.~Anand, J.~S. Ovall, and C.~Turc.
\newblock Well-conditioned boundary integral equations for two-dimensional
  sound-hard scattering problems in domains with corners.
\newblock {\em J. Integral Equations Appl.}, 24(3):321--358, 2012.

\bibitem{antoine2008numerical}
Xavier Antoine, Chokri Chniti, and Karim Ramdani.
\newblock On the numerical approximation of high-frequency acoustic multiple
  scattering problems by circular cylinders.
\newblock {\em Journal of Computational Physics}, 227(3):1754--1771, 2008.

\bibitem{antoine2012wide}
Xavier Antoine, Karim Ramdani, and Bertrand Thierry.
\newblock Wide frequency band numerical approaches for multiple scattering
  problems by disks.
\newblock {\em Journal of Algorithms \& Computational Technology},
  6(2):241--260, 2012.

\bibitem{bendali}
Nolwenn Balin, Abderrahmane Bendali, and Francis Collino.
\newblock Domain decomposition and additive schwarz techniques in the solution
  of a te model of the scattering by an electrically deep cavity.
\newblock In {\em Domain Decomposition Methods in Science and Engineering},
  pages 149--156. Springer, 2005.

\bibitem{boubendirDDM}
Y.~Boubendir, X.~Antoine, and C.~Geuzaine.
\newblock A quasi-optimal non-overlapping domain decomposition algorithm for
  the {H}elmholtz equation.
\newblock {\em J. Comput. Phys.}, 231(2):262--280, 2012.

\bibitem{turc2}
Y.~Boubendir, O.~Bruno, C.~Levadoux, and C.~Turc.
\newblock Integral equations requiring small numbers of {K}rylov-subspace
  iterations for two-dimensional smooth penetrable scattering problems.
\newblock {\em Appl. Numer. Math.}, 95:82--98, 2015.

\bibitem{turc1}
Y.~Boubendir and C.~Turc.
\newblock Wave-number estimates for regularized combined field boundary
  integral operators in acoustic scattering problems with neumann boundary
  conditions.
\newblock {\em IMA Journal of Numerical Analysis}, 33(4):1176--1225, 2013.

\bibitem{lintner}
Oscar~P Bruno and St{\'e}phane~K Lintner.
\newblock A high-order integral solver for scalar problems of diffraction by
  screens and apertures in three-dimensional space.
\newblock {\em Journal of Computational Physics}, 252:250--274, 2013.

\bibitem{KressColton}
D.~Colton and R.~Kress.
\newblock {\em Integral equation methods in scattering theory}.
\newblock Pure and Applied Mathematics (New York). John Wiley \& Sons Inc., New
  York, 1983.
\newblock A Wiley-Interscience Publication.

\bibitem{Depres}
Bruno Despr{\'e}s.
\newblock D\'ecomposition de domaine et probl\`eme de {H}elmholtz.
\newblock {\em C. R. Acad. Sci. Paris S\'er. I Math.}, 311(6):313--316, 1990.

\bibitem{dominguez2015well}
Victor Dominguez, Mark Lyon, and Catalin Turc.
\newblock Well-posed boundary integral equation formulations and
  nystr$\backslash$" om discretizations for the solution of helmholtz
  transmission problems in two-dimensional lipschitz domains.
\newblock {\em arXiv preprint arXiv:1509.04415}, 2015.

\bibitem{Duff:1983:MSI:356044.356047}
I.~S. Duff and J.~K. Reid.
\newblock The multifrontal solution of indefinite sparse symmetric linear.
\newblock {\em ACM Trans. Math. Softw.}, 9(3):302--325, September 1983.

\bibitem{EsFaVer:1992}
L.~Escauriaza, E.~B. Fabes, and G.~Verchota.
\newblock On a regularity theorem for weak solutions to transmission problems
  with internal {L}ipschitz boundaries.
\newblock {\em Proc. Amer. Math. Soc.}, 115(4):1069--1076, 1992.

\bibitem{Fliss1}
Sonia Fliss, Dirk Klindworth, and Kersten Schmidt.
\newblock Robin-to-{R}obin transparent boundary conditions for the computation
  of guided modes in photonic crystal wave-guides.
\newblock {\em BIT}, 55(1):81--115, 2015.

\bibitem{foldy1945multiple}
Leslie~L Foldy.
\newblock The multiple scattering of waves. i. general theory of isotropic
  scattering by randomly distributed scatterers.
\newblock {\em Physical Review}, 67(3-4):107, 1945.

\bibitem{Gander1}
Martin~J. Gander, Fr{\'e}d{\'e}ric Magoul{\`e}s, and Fr{\'e}d{\'e}ric Nataf.
\newblock Optimized {S}chwarz methods without overlap for the {H}elmholtz
  equation.
\newblock {\em SIAM J. Sci. Comput.}, 24(1):38--60 (electronic), 2002.

\bibitem{ganesh2012convergence}
M~Ganesh, SC~Hawkins, and R~Hiptmair.
\newblock Convergence analysis with parameter estimates for a reduced basis
  acoustic scattering t-matrix method.
\newblock {\em IMA Journal of Numerical Analysis}, page drr041, 2012.

\bibitem{George:1973:NDR}
Alan George.
\newblock Nested dissection of a regular finite element mesh.
\newblock 10(2):345--363, April 1973.

\bibitem{Gillman1}
Adrianna Gillman, Alex~H. Barnett, and Per-Gunnar Martinsson.
\newblock A spectrally accurate direct solution technique for frequency-domain
  scattering problems with variable media.
\newblock {\em BIT}, 55(1):141--170, 2015.

\bibitem{gurel1992recursive}
L~G{\"u}rel and WC~Chew.
\newblock A recursive t-matrix algorithm for strips and patches.
\newblock {\em Radio science}, 27(3):387--401, 1992.

\bibitem{Fliss2}
Patrick Joly, Jing-Rebecca Li, and Sonia Fliss.
\newblock Exact boundary conditions for periodic waveguides containing a local
  perturbation.
\newblock {\em Commun. Comput. Phys}, 1(6):945--973, 2006.

\bibitem{kopriva1998staggered}
David~A Kopriva.
\newblock A staggered-grid multidomain spectral method for the compressible
  navier--stokes equations.
\newblock {\em Journal of Computational Physics}, 143(1):125--158, 1998.

\bibitem{KressCorner}
R.~Kress.
\newblock A {N}ystr\"om method for boundary integral equations in domains with
  corners.
\newblock {\em Numer. Math.}, 58(2):145--161, 1990.

\bibitem{kusmaul}
R.~Kussmaul.
\newblock Ein numerisches {V}erfahren zur {L}\"osung des {N}eumannschen
  {N}eumannschen {A}ussenraumproblems f\"ur die {H}elmholtzsche
  {S}chwingungsgleichung.
\newblock {\em Computing (Arch. Elektron. Rechnen)}, 4:246--273, 1969.

\bibitem{lai2014fast}
Jun Lai, Motoki Kobayashi, and Leslie Greengard.
\newblock A fast solver for multi-particle scattering in a layered medium.
\newblock {\em Optics express}, 22(17):20481--20499, 2014.

\bibitem{lax1951multiple}
Melvin Lax.
\newblock Multiple scattering of waves.
\newblock {\em Reviews of Modern Physics}, 23(4):287, 1951.

\bibitem{martensen}
{E.} Martensen.
\newblock \"{U}ber eine {M}ethode zum r\"aumlichen {N}eumannschen {P}roblem mit
  einer {A}nwendung f\"ur torusartige {B}erandungen.
\newblock {\em Acta Math.}, 109:75--135, 1963.

\bibitem{martin2006multiple}
Paul~A Martin.
\newblock {\em Multiple scattering: interaction of time-harmonic waves with N
  obstacles}, volume 107.
\newblock Cambridge University Press, 2006.

\bibitem{mclean:2000}
W.~McLean.
\newblock {\em Strongly elliptic systems and boundary integral equations}.
\newblock Cambridge University Press, Cambridge, 2000.

\bibitem{Bennetts1}
Fabien Montiel, Vernon~A. Squire, and Luke~G. Bennetts.
\newblock Evolution of directional wave spectra through finite regular and
  randomly perturbed arrays of scatterers.
\newblock {\em SIAM J. Appl. Math.}, 75(2):630--651, 2015.

\bibitem{Nataf}
Fr{\'e}d{\'e}ric Nataf.
\newblock Interface connections in domain decomposition methods.
\newblock In {\em Modern methods in scientific computing and applications
  ({M}ontr\'eal, {QC}, 2001)}, volume~75 of {\em NATO Sci. Ser. II Math. Phys.
  Chem.}, pages 323--364. Kluwer Acad. Publ., Dordrecht, 2002.

\bibitem{norris2013stable}
Andrew~N Norris, Adam~J Nagy, and Feruza~A Amirkulova.
\newblock Stable methods to solve the impedance matrix for radially
  inhomogeneous cylindrically anisotropic structures.
\newblock {\em Journal of Sound and Vibration}, 332(10):2520--2531, 2013.

\bibitem{orszag1980spectral}
Steven~A Orszag.
\newblock Spectral methods for problems in complex geometries.
\newblock {\em Journal of Computational Physics}, 37(1):70--92, 1980.

\bibitem{pfeiffer2003multidomain}
Harald~P Pfeiffer, Lawrence~E Kidder, Mark~A Scheel, and Saul~A Teukolsky.
\newblock A multidomain spectral method for solving elliptic equations.
\newblock {\em Computer physics communications}, 152(3):253--273, 2003.

\bibitem{steinbach2011stable}
O~Steinbach and M~Windisch.
\newblock Stable boundary element domain decomposition methods for the
  helmholtz equation.
\newblock {\em Numerische Mathematik}, 118(1):171--195, 2011.

\bibitem{turc2016well}
Catalin Turc, Yassine Boubendir, and Mohamed~Kamel Riahi.
\newblock Well-conditioned boundary integral equation formulations and
  nystr$\backslash$" om discretizations for the solution of helmholtz problems
  with impedance boundary conditions in two-dimensional lipschitz domains.
\newblock {\em arXiv preprint arXiv:1607.00769}, 2016.

\bibitem{waterman1965matrix}
PC~Waterman.
\newblock Matrix formulation of electromagnetic scattering.
\newblock {\em Proceedings of the IEEE}, 53(8):805--812, 1965.

\end{thebibliography}

\end{document}